\documentclass[a4paper,12pt]{amsart}

\usepackage{environ}
\usepackage{fancyhdr}
\usepackage{mabliautoref}
\usepackage{amssymb,amsthm,amsmath}
\RequirePackage[dvipsnames,usenames]{xcolor}
\usepackage{hyperref}
\usepackage{cleveref}
\usepackage{mathtools}
\usepackage{tcolorbox}
\usepackage[all]{xy}
\usepackage{tikz}
\usepackage{tikz-cd}

\makeatletter
\def\fixtikzforbreqn#1#2{%
	\protected\edef#1{\noexpand\ifmmode\mathchar\the\mathcode`#2 \noexpand\else#2\noexpand\fi}%
}
\fixtikzforbreqn\tikz@nonactivesemicolon;
\fixtikzforbreqn\tikz@nonactivecolon:
\fixtikzforbreqn\tikz@nonactivebar|
\fixtikzforbreqn\tikz@nonactiveexlmark!
\makeatother

\usetikzlibrary{decorations.markings,shapes,positioning}
\usepackage{enumitem}
\usepackage{stmaryrd}
\usepackage{xfrac}
\usepackage{fix-cm}
\usepackage{faktor}
\usepackage{minibox}
\usepackage{commath}
\usepackage{chngcntr}
\usepackage{setspace}
\usepackage{array}
\usepackage[toc,page]{appendix}
\usepackage{calligra,mathrsfs}
\usepackage{mathtools}
\usepackage{bbm}

\hypersetup{
	bookmarks,
	bookmarksdepth=3,
	bookmarksopen,
	bookmarksnumbered,
	pdfstartview=FitH,
	colorlinks,backref,hyperindex,
	linkcolor=Sepia,
	anchorcolor=BurntOrange,
	citecolor=Cyan,
	filecolor=BlueViolet,
	menucolor=Yellow,
	urlcolor=OliveGreen
}

\newtcolorbox{activitybox}[1][]{%
	breakable,
	enhanced,
	colback=lightergray,
	boxrule=3pt,
	arc=5pt,
	outer arc=5pt,
	boxsep=10pt,
	colframe=darkergray,
	coltitle=white,
	#1
}

\newcommand{\quot}[2]{%
	\raise1ex\hbox{$#1$}\Big/\lower1ex\hbox{$#2$}%
}

\newcommand{\colim}{\varinjlim}
\renewcommand{\lim}{\varprojlim}

\makeatletter\def\Rvarlim@#1#2{%
	\vtop{\m@th\ialign{##\cr
			\hfil$#1\operator@font Rlim$\hfil\cr
			\noalign{\nointerlineskip\kern1.5\ex@}#2\cr
			\noalign{\nointerlineskip\kern-\ex@}\cr}}%
}
\makeatother

\makeatletter\def\Rlim{%
	\mathop{\mathpalette\Rvarlim@{\leftarrowfill@\textstyle}}\nmlimits@
}
\makeatother

\newcommand{\expl}[2]{\underset{\mathclap{\minibox[c]{$\uparrow$\\ \fbox{\footnotesize #2}}}}{#1}}

\newcommand{\ra}{\rightarrow}

\def\xto#1#2{\xrightarrow{#1}{#2}}
\def\xto#1{\xrightarrow{#1}}
\newcommand{\surj}{\twoheadrightarrow}
\newcommand{\inj}{\hookrightarrow}
\newcommand{\ttilde}{\widetilde}

\newcommand{\HHom}{\mathcal{H}om}

\newcommand{\et}{\acute{e}t}

\newcommand{\stacksproj}[1]{{\cite[Tag~\href{http://stacks.math.columbia.edu/tag/#1}{#1}]{Stacks_Project}}}

\newcommand{\inc}{\subseteq}

\newcommand{\cni}{\supseteq}

\newcommand{\esp}{\mbox{ }}
\newcommand{\bighat}{\widehat}

\DeclareMathAlphabet{\mathchanc}{OT1}{pzc}%
{m}{it}

\renewcommand{\set}[2]{\{ \ #1 \  | \ #2 \ \}}

\newcommand{\bF}{\mathbb{F}}

\newcommand{\bQ}{\mathbb{Q}}

\newcommand{\bZ}{\mathbb{Z}}

\newcommand{\scr}{\mathcal}
\newcommand{\cA}{\scr{A}}
\newcommand{\cB}{\scr{B}}
\newcommand{\cC}{\scr{C}}

\newcommand{\cF}{\scr{F}}
\newcommand{\cG}{\scr{G}}
\newcommand{\cH}{\scr{H}}

\newcommand{\cM}{\scr{M}}
\newcommand{\cN}{\scr{N}}
\newcommand{\cO}{\scr{O}}
\newcommand{\cP}{\scr{P}}

\newcommand{\cR}{\scr{R}}

\newcommand{\cT}{\scr{T}}

\newcommand{\cV}{\scr{V}}

\DeclareMathOperator{\injj}{{inj}}

\DeclareMathOperator{\nilp}{{nilp}}

\DeclareMathOperator{\FM}{{FM}}

\DeclareMathOperator{\alb}{{alb}}
\DeclareMathOperator{\Alb}{{Alb}}

\DeclareMathOperator{\codim}{codim}
\DeclareMathOperator{\coker}{{coker}}

\DeclareMathOperator{\depth}{{depth}}

\DeclareMathOperator{\Ext}{Ext}

\DeclareMathOperator{\Hom}{Hom}

\DeclareMathOperator{\Tor}{Tor}

\DeclareMathOperator{\im}{{im}}

\DeclareMathOperator{\length}{{length}}

\DeclareMathOperator{\Pic}{Pic}

\DeclareMathOperator{\projdim}{{projdim}}

\DeclareMathOperator{\rank}{{rank}}

\DeclareMathOperator{\red}{red}

\DeclareMathOperator{\Spec}{{Spec}}

\DeclareMathOperator{\Supp}{{Supp}}

\DeclareMathOperator{\Crys}{Crys}
\DeclareMathOperator{\coh}{coh}

\DeclareMathOperator{\RGamma}{R\Gamma}

\DeclareMathOperator{\Coh}{Coh}

\DeclareMathOperator{\ShHom}{\mathscr{H}\text{\kern -3pt {\calligra\large om}}\,}

\DeclareFontFamily{OT1}{pzc}{}
\DeclareFontShape{OT1}{pzc}{m}{it}{<-> s * [1.200] pzcmi7t}{}
\DeclareMathAlphabet{\mathpzc}{OT1}{pzc}{m}{it}

\AtBeginDocument{%
	\mathchardef\phialt=\phi
	\mathchardef\phi=\varphi
}

\newcommand{\factor}[2]{\left. \raise 2pt\hbox{\ensuremath{#1}} \right/
	\hskip -2pt\raise -2pt\hbox{\ensuremath{#2}}}

\makeatletter
\renewcommand\subsection{
	\renewcommand{\sfdefault}{pag}
	\@startsection{subsection}%
	{2}{0pt}{.8\baselineskip}{.4\baselineskip}{\raggedright
		\sffamily\itshape\small\bfseries
}}
\renewcommand\section{
	\renewcommand{\sfdefault}{phv}
	\@startsection{section} %
	{1}{0pt}{\baselineskip}{.8\baselineskip}{\centering
		\sffamily
		\scshape
		\bfseries
}}
\makeatother

\definecolor{gr}{rgb}{0,0.5,0}

\newcommand{\fm}{\mathfrak{m}}

\newcommand{\Addresses}{{
		\bigskip
		\footnotesize
		
		\textsc{\'Ecole Polytechnique F\'ed\'erale de Lausanne, SB MATH CAG, MA C3 615 (B\^atiment MA), Station 8, CH-1015 Lausanne, Switzerland}\par\nopagebreak
		\textit{E-mail address}: \texttt{jefferson.baudin@epfl.ch}

}}

\author{Jefferson Baudin}
\date{}

\usepackage[margin=.95in]{geometry}
\setlist{  
	listparindent=\parindent,
	parsep=0pt,
}


\title{Generic vanishing theory in positive characteristic}
\begin{document}
	\maketitle
\begin{abstract}
	We simplify and improve the main fundamental theorems of positive characteristic generic vanishing theory. As a quick corollary of the theory, we prove that a normal proper variety $X$ of maximal Albanese dimension satisfies $H^0(X, \omega_X) \neq 0$. If $\Alb(X)$ is in addition ordinary, then $S^0(X, \omega_X) \neq 0$.
\end{abstract}	
	\tableofcontents
\section{Introduction}

\subsection{Background from characteristic zero}

Generic vanishing techniques are very powerful tools in the study of irregular varieties in characteristic zero. For example, they allow for birational characterizations of abelian varieties \cite{Chen_Hacon_Characterization_of_abelian_varieties, Hacon_Pardini_On_the_birational_geometry_of_vars_of_mad, Jiang_An_effective_version_of_a_thm_of_Kawamata_on_the_Albanese_map, Pareschi_Basic_results_on_irr_vars_via_FM_methods}, a deep effective understanding of pluricanonical systems \cite{Jiang_Lahoz_Tirabassi_On_the_Iitaka_fibration_of_varieties_of_maximal_albanese_dimension, Barja_Lahoz_Naranjo_Pareschi_On_the_bicanonical_map_of_irregular_varieties, Chen_Jiang_Positivity_in_varieties_of_maximal_Albanese_dimension}, a proof to Iitaka's conjecture in certain cases \cite{Cao_Paun_Kodaira_dimension_of_algebraic_fiber_spaces_over_abelian_varieties, Hacon_Popa_Schnell_Alg_fiber_spaces_over_abelian_varieties, Meng_Popa_Kodaira_dimension_of_fibrations_over_abelian_varieties} and so on \cite{Ein_Lazarsfeld_Singularities_of_theta_divisors_and_the_birational_geometry_of_irregular_varieties, Hacon_Pardini_Birational_Characterization_Of_Products_Of_Curves_Of_Genus_2, Chen_Debarre_Jiang_Varieties_with_vanishing_euler_char, Pareschi_Popa_Regularity_on_AV_I, Pareschi_Popa_Regularity_on_AV_III, Jiang_Pareschi_Cohomological_rank_functions_on_abelian_varieties}. 

Let us explain what the main generic vanishing theorem is in characteristic zero, from which the theory has evolved. Fix an abelian variety $A$ over an algebraically closed field, and let $\bighat{A} = \Pic^0(A)$ be its dual. Given a coherent sheaf $\cM$ on $A$ , define the \emph{cohomological support loci} \[ V^i(\cM) \coloneqq \set{\alpha \in \bighat{A}}{H^i(A, \cM \otimes \alpha) \neq 0} \inc \bighat{A}. \] Using the upper-semicontinuity of the dimension of cohomology groups in families (see e.g. \cite[Theorem 12.8]{Hartshorne_Algebraic_Geometry}), one can show that these subsets are closed. The point of the generic vanishing theorem is to say that for some specific sheaves on $A$ of geometric interest, these subsets $V^i$ are rather small as $i$ increases. With this in mind, Pareschi and Popa defined in \cite{Pareschi_Popa_GV_sheaves_FM_transform_and_generic_vanishing} the following notion:

\begin{definition*}
	We say that a coherent sheaf $\cM$ on $A$ is a \emph{GV--sheaf} if for all $i \geq 0$, \[ \codim V^i(\cM) \geq i. \]
\end{definition*}
The letters ``GV'' stand for \emph{generic vanishing}: if $\cM$ is a GV--sheaf and $\alpha \in \bighat{A}$ is a general line bundle, then in particular $H^i(A, \cM \otimes \alpha) = 0$ for all $i > 0$ (i.e. ``cohomology groups vanish generically''). One can therefore see the property of being a GV--sheaf as a positivity property (for example, a line bundle on an abelian variety is a GV--sheaf if and only if it is numerically equivalent to an effective divisor, see \autoref{example_when_line_bundles_are_GV}). \\

A general philoshophy is that given a morphism $f \colon X \to Y$ of smooth complex varieties, then the sheaves $R^jf_*\omega_{X/Y}$ and $f_*\omega_{X/Y}^{\otimes m}$ satisfy certain positivity properties for $i \geq 0$ and $m \geq 1$ (see e.g. \cite{Griffiths_Periods_of_integral_on_algebraic_manifolds_III, Viehweg_Weak_positivity_and_the_additivity_of_Kodaira_dimension_I, Kollar_Higher_Direct_Image_of_Dualizing_Sheaves_I, Hacon_Popa_Schnell_Alg_fiber_spaces_over_abelian_varieties}). When $Y$ is an abelian variety, we have that $\omega_{X/Y} \cong \omega_X$, since $\omega_Y \cong \cO_Y$, so the following fundamental theorems fit neatly in this philosophy:

\begin{theorem*}[\cite{Green_Lazarsfeld_Generic_vanishing, Green_Lazarsfeld_GV_2, Simpson_Subspaces_of_moduli_spaces_of_rank_one_local_systems, Hacon_A_derived_category_approach_to_generic_vanishing, Hacon_Popa_Schnell_Alg_fiber_spaces_over_abelian_varieties}]
	Let $X$ be a smooth projective complex variety, and let $a \colon X \to A$ be a morphism to an abelian variety. Then each $R^ja_*\omega_X$ is a GV-sheaf. Furthermore, the subsets $V^i(R^ja_*\omega_X)$ are finite union of torsion translates of abelian subvarieties of $\bighat{A}$. The same statements also hold for the sheaves $a_*\omega_X^{\otimes m}$ with $m \geq 1$ instead.
\end{theorem*}

In \cite{Mukai_Fourier_Mukai_transform}, Mukai defined an equivalence of categories between the derived categories of coherent sheaves $D^b(A)$ and $D^b(\bighat{A})$, called the \emph{Fourier-Mukai transform}. We can modify it to obtain the relevant equivalence to us: the \emph{symmetric Fourier-Mukai transform} (see \cite{Schnell_Fourier_Mukai_transform_made_easy}). It is denoted 
\[ \begin{tikzcd}
	D^b_{\coh}(A)^{op} \arrow[r, "\FM_A"] & D^b_{\coh}(\bighat{A}).
\end{tikzcd} \] Even if we start with a sheaf $\cM$ on $A$, its image $\FM_A(\cM)$ might be a complex in the derived category, not just a sheaf! A breakthrough of Hacon and Pareschi--Popa was to realize the following:

\begin{theorem*}\label{GV_char_zero_intro}\cite{Hacon_A_derived_category_approach_to_generic_vanishing, Pareschi_Popa_GV_sheaves_FM_transform_and_generic_vanishing}
	A sheaf $\cM$ on an abelian variety $A$ is a GV-sheaf if and only if $\FM_A(\cM)$ is concentrated in degre zero, i.e.  we have \[ \cH^i\FM_A(\cF) = 0 \] for all $i \neq 0$.
\end{theorem*}

Playing explicitly with Fourier--Mukai transforms and understand explicitly commutative algebraic properties (e.g. depth, rank, length, support) of $\FM_A(\cF)$ for certain GV--sheaves $\cF$ such as $a_*\omega_X$ is central in many of the proofs that use generic vanishing.
	
\subsection{The positive characteristic story}

Over fields of prime characteristic, the generic vanishing theorem above is known to fail (\cite{Filipazzi_GV_fails_in_pos_char}). The key idea of Hacon and Patakfalvi in \cite{Hacon_Pat_GV_Characterization_Ordinary_AV} is to ask for a weakening of this vanishing, namely we want it to hold \emph{up to nilpotence}.

Let us explain what we mean: Cartier defined in \cite{Cartier_une_nouvelle_operation_sur_les_formes_differentielles} an important operator $C$ on the sheaf of top forms $\omega_X$, for $X$ a smooth projective variety over a perfect field $k$ of positive characteristic (see \autoref{ex:Cartier_operator}). This operation is additive and $p^{-1}$--linear, meaning that for all $\lambda \in \cO_X$ and $\eta \in \omega_X$, \[ C(\lambda^p\eta) = \lambda C(\eta). \]

In other words, we have a morphism $C \colon F_*\omega_X \to \omega_X$, where $F$ denotes the Frobenius on $X$. In general, the data of a coherent sheaf $\cM$ with an $\cO_X$--linear homomorphism $F_*\cM \to \cM$ is called a \emph{Cartier module}.

Let us give an example why working up to nilpotence allows for pleasant vanishing properties. Although Kodaira vanishing fails in positive characteristic (\cite{Raynaud_Failure_Kodaira_vanishing}), it \emph{does} hold up to nilpotence under this Cartier operator in the following sense: for any ample line bundle $L$, there exists $e > 0$ such that for all $i > 0$, 
\begin{equation}\label{KV_up_to_nilpotence}
	H^i(X, F^e_*\omega_X \otimes L) \to H^i(X, \omega_X \otimes L)
\end{equation} is zero. This is immediate, since for $e \gg 0$, \[ H^i(X, F^e_*\omega_X \otimes L) \cong H^i(X, F^e_*(\omega_X \otimes F^{e, *}L) = H^i(X, F^e_*(\omega_X \otimes L^{p^e})) = 0  \] by Serre vanishing.

Although this is a very basic observation, it appears a lot in positive characteristic birational geometry \cite{Schwede_A_canonical_linear_system_associated_to_adjoint_divisors_in_char_p}. This is for example used in the study of positive characteristic fibrations (\cite{Patakfalvi_On_Subadditivity_of_Kodaira_dimension_in_positive_characteristic, Ejiri_Patakfalvi_The_Demailly_Peternell_Schneider_conjecture_is_true_in_pos_char, Ejiri_Numerical_Kodaira_dimension_of_algebraic_fiber_spaces_in_positive_characteristic}), or in the development of the three--dimensional minimal model program (\cite{Hacon_Xu_On_the_3dim_MMP_in_pos_char}).

One might wonder if other kinds of vanishing theorems up to nilpotence could hold. The author's experience suggests that an accurate philosophy is this one:

\begin{itemize}
	\item if some vanishing follows from Kodaira vanishing in characteristic zero, then it should hold when working up to nilpotence;
	\item if some vanishing follows from Kawamata--Viehweg vanishing (here we mean Kodaira vanishing for big and nef line bundles) in characteristic zero, then it should not hold when working up to nilpotence in general.
\end{itemize}

Here are some instances of the above philosophy: 

\begin{enumerate}
	\item (\cite[Lemma 2.1]{Hacon_Witaszek_On_the_relative_MMP_for_4folds_in_pos_and_mixed_char}, \cite[Proposition 2.10]{Baudin_Patakfalvi_Rosler_Zdanowicz_On_Gorenstein_Q_p_rational_threefolds_and_fourfolds}) log canonical centers of dlt pairs are $S_2$ up to universal homeomorphism;
	\item (\cite{Baudin_Bernasconi_Kawakami_Frobenius_GR_fails}) for $n \leq 4$, klt $n$--fold singularities are $\bF_p$--rational, meaning that if $\pi \colon Y \to X$ is a resolution of such a singularity, then for all $i > 0$, the action of the Frobenius on $R^i\pi_*\cO_Y$ is nilpotent;
	\item(\cite{Baudin_Bernasconi_Kawakami_Frobenius_GR_fails,Totaro_The_failure_of_KV_and_terminal_singularities_that_are_not_CM, Totaro_Terminal_3folds_that_are_not_CM}) already klt 3--fold singularities might not be Cohen--Macaulay up to nilpotence (the Frobenius action on the second local cohomology group might not be nilpotent) in low characteristics;
	\item (\cite{Baudin_Bernasconi_Kawakami_Frobenius_GR_fails}) the vanishing \autoref{KV_up_to_nilpotence} can fail if we replace $L$ ample by $L$ semiample and big in any characteristic, even for smooth projetive threefolds (nevertheless, see \cite[Theorem 1.6]{Bhatt_Derived_Splinters_In_Positive_Characteristic}).
\end{enumerate}

An other instance of this philosophy is the characteristic zero generic vanishing theorem, which ultimately relies a Kodaira vanishing--type theorem of Koll\'ar (\cite{Kollar_Higher_Direct_Image_of_Dualizing_Sheaves_I, Kollar_Higher_Direct_Image_of_Dualizing_Sheaves_II}). Here is the positive characteristic replacement, that roughly states that \[ \mbox{``Cartier modules are GV--sheaves up to nilpotence.''} \]

\begin{theorem*}[{\cite[Corollary 3.1.4]{Hacon_Pat_GV_Characterization_Ordinary_AV}}]
	Let $\cM$ be a Cartier module on an abelian variety $A$ (e.g. $R^ja_*\omega_X$, where $a \colon X \to A$ is a morphism from a smooth projective variety). Then for all $i \neq 0$, the induced action on each $\cH^i\FM_A(\cM)$ is nilpotent.
\end{theorem*}

Let $\cM$ be a Cartier module on $A$. Since the above theorem suggests that we should work up to nilpotence, one should define version of $V^i(\cM)$ up to nilpotence. The definition is as follows \[ W^i_F(\cM) \coloneqq \set{\alpha \in \bighat{A}}{\lim_e H^i(A, F^e_*\cM \otimes \cP_{\alpha}) \neq 0}. \] Since $(p^{-1})W^i_F(\cM) = W^i_F(\cM)$ by the projection formula, these subsets are not closed in general (e.g. if it is non--empty, then it is in fact dense when in case where $A$ is ordinary). The theorem below shows that one can approximate $W^i_F(\cM)$ by canonically defined closed subsets with the expected properties.

\begin{thm_letter}[{\autoref{generic vanishing}, \cite[Theorem 1.1]{Hacon_Pat_GV_Geom_Theta_Divs}}]\label{intro_SCSL}
	Let $\cM$ be a Cartier module on $A$, and set $g \coloneqq \dim(A)$. Then for all $0 \leq i \leq g$, there exist closed subsets $W^i$ and $Z^i$ of $\bighat{A}$ with the following properties:
	\begin{enumerate}[topsep=1ex, itemsep=1ex]
		\item $W^{g} \inc W^{g - 1} \inc \dots \inc W^1 \inc  W^0$;
		\item $\codim W^i \geq i$;
		\item $p(W^i) = W^i$. In particular, if $A$ has no supersingular factor, then $W^i$ is a finite union of torsion translates of abelian subvarieties of $\bighat{A}$;
		\item we have \[ W^0 \inc \set{\alpha \in \bighat{A}}{H^0(A, \cM \otimes \cP_{\alpha}) \neq 0}; \]
		\item $\codim(Z^i) \geq i + 1$;
		\item there are inclusions  \[\left(\bigcup_{e \geq 0}(p^{es})^{-1}W^i\right) \setminus \ttilde{Z^i} \: \inc \: W^i_F(\cM) \: \inc \: \bigcup_{e \geq 0}(p^{es})^{-1}W^i,\] where \[ W^i_F(\cM) \coloneqq \set{\alpha \in \bighat{A}}{\lim_e H^i(A, F^e_*\cM \otimes \cP_{\alpha}) \neq 0} \] and \[ \ttilde{Z^i} \coloneqq \set{\alpha \in \bighat{A}}{p^{es}(\alpha) \in Z^i \mbox{ for infinitely many }e}. \]
		In particular, if $W^i$ has a component of exact codimension $i$, then very general points of $W^i$ lie in $W^i_F(\cM)$.
	\end{enumerate}
\end{thm_letter}

This theorem is a slight improvement of \cite[Theorem 1.1]{Hacon_Pat_GV_Geom_Theta_Divs}. Nevertheless, this small addition is essential in the proof of \cite[Theorem 6.3.2]{Baudin_Duality_between_perverse_sheaves_and_Cartier_crystals}. \\

In \autoref{sec:examples}, we give some examples illustrating the theory, and showing the kind of pathologies that can occur.

\subsection{Goal of the paper and new results}

\setcounter{theorem}{1}

An issue with the literature in positive characteristic generic vanishing is that it is scattered in several different papers, and might be hard to follow sometimes. Our goal here is to write everything in one single reference and simplify the proofs, with the hope of making the subject more accessible to a broader audience.

Therefore, we first decided to reprove some of the main statements of positive characteristic generic vanishing theory, namely \cite[Corollary 3.1.4]{Hacon_Pat_GV_Characterization_Ordinary_AV} and \cite[Theorems 5.2 and 1.1]{Hacon_Pat_GV_Geom_Theta_Divs}. We already mentioned \cite[Theorems 1.1]{Hacon_Pat_GV_Geom_Theta_Divs}, so let us focus on the other two main statements. Although our proof of \cite[Corollary 3.1.4]{Hacon_Pat_GV_Characterization_Ordinary_AV} is virtually the same, and is only here for the sake of completeness, we believe that our proof of \cite[Theorem 5.2]{Hacon_Pat_GV_Geom_Theta_Divs} is significantly easier than the original one. As it turns out, we pushed the study to obtain the following new result, which gives a lot of flexibility and which will be used in \cite{Baudin_Effective_Characterization_of_ordinary_abelian_varieties}:

\begin{thm_letter}[{\autoref{equivalence_V_crystals_and_Cartier_crystals}}]\label{intro_eq_cat}
	Let $A$ be an abelian variety. Then the functor $\FM_A$ induces an equivalence of categories between Cartier crystals and $V$--crystals.
\end{thm_letter}

Let us explain the setup. We already explained that the relevant objects satisfying generic vanishing are Cartier modules up to nilpotence. One can make precise sense of ``modding out the category of Cartier modules by nilpotent ones'', and this gives rise to so--called \emph{Cartier crystals}.

As it turns out, given a Cartier module $\cM$, we obtain for each $i \in \bZ$ a natural morphism \[ \cH^i(\FM_A(\cM)) \to V^*\cH^i(\FM_A(\cM)), \] where $V$ denotes the Verschiebung (the isogeny dual to the Frobenius). Such an object is called a \emph{$V$--module}, and was first defined in \cite{Hacon_Pat_GV_Geom_Theta_Divs}. Since we work with Cartier modules up to nilpotence (i.e. Cartier crystals), we should also treat $V$--modules up to nilpotence, that is \emph{$V$--crystals}. With the idea in mind that Cartier crystals are GV--sheaves, the equivalence of categories is then given by \[ \cM \mapsto \cH^0\FM_A(\cM). \]
We refer the reader to \autoref{sec:equiv_cat} for further details. \\ 

Let us finish this introduction with a new concrete application of generic vanishing. In characteristic zero, one can show that if $X$ is normal proper variety of maximal Albanese dimension, then $H^0(X, \omega_X) \neq 0$ (see \autoref{last_rem}). In positive characteristic however, one cannot run such an argument due to the presence of purely inseparable phenomena. Nevertheless, we can use generic vanishing theory to obtain the following:

\begin{thm_letter}[{\autoref{effectivity_canonical}}]\label{intro_effectivity_canonical}
	Let $X$ be a normal proper variety of maximal Albanese dimension. Then $H^0(X, \omega_X) \neq 0$. If $X$ further admits a generically finite morphism to an ordinary abelian variety, then $S^0(X, \omega_X) \neq 0$.
\end{thm_letter}

The vector space $S^0(X, \omega_X) \inc H^0(X, \omega_X)$ is defined as \[  S^0(X, \omega_X) \coloneqq \bigcap_{e > 0}\im(H^0(X, F^e_*\omega_X) \to H^0(X, \omega_X)), \] where $F^e_*\omega_X \to \omega_X$ denotes the Cartier operator iterated $e$ times. An equivalent way to state that $S^0(X, \omega_X) \neq 0$ is that there exists $0 \neq \eta \in H^0(X, \omega_X)$ such that $C(\eta) = \eta$, where $C$ denotes the Cartier operator.

This theorem and its proof will be an important first step to obtain an effective birational characterization of abelian varieties, see \cite{Baudin_Effective_Characterization_of_ordinary_abelian_varieties}.

\subsection{Acknowledgments}
I would like to thank Fabio Bernasconi, Lucas Gerth, Léo Navarro Chafloque and Zsolt Patakfalvi for insightful discussions about the content of this article. I also thank the anonymous referee for their careful reading of the article and their insightful comments.

Financial support was provided by grant $\#$200020B/192035 from the Swiss National Science Foundation (FNS/SNF), and by grant $\#$804334 from the European Research Council (ERC).

\subsection{Notations}

\begin{itemize}
	\item We fix once and for all a prime number $p > 0$. All rings and schemes in this paper are defined over $\bF_p$.
	\item We fix an algebraically closed field $k$ (of characteristic $p$).
	\item A \emph{point} denotes a closed point. We will refer to non--closed point as \emph{scheme--theoretic points}.
	\item Throughout, $A$ denotes an abelian variety over $k$ of dimension $g$, with dual abelian variety $\bighat{A}$ and normalized Poincaré bundle $\cP \in \Pic(A \times \bighat{A})$. For a point $\alpha \in \bighat{A}$, we write $\cP_\alpha \coloneqq \cP|_{A \times \alpha} \in \Pic^0(A)$. Given an integer $n \in \bZ$, multiplication by $n$ on the group law of $A$ will still be denoted $n$.
	\item Given a subset $S \inc A$, we denote by $-S$ the image of $S$ under the inversion involution on $A$.
	\item A variety is a scheme $X$ which is separated and of finite type over $k$. 
	\item The symbol $F$ always denotes the absolute Frobenius of a $\bF_p$--scheme, i.e. the unique endomorphism with restrict to taking $p$--powers on affine charts.
	\item An $\bF_p$--scheme $X$ is $F$--finite if its absolute Frobenius morphism $F \colon X \to X$ is finite. 
	
	\item If $\pi \colon X \to S$ is a morphism, then $F_{X/S} \colon X \to X^{(p)}$ always denotes the relative Frobenius over $S$, i.e. the unique map fitting in the following diagram: 
	\[ \begin{tikzcd}
		X \arrow[rrrrd, "F", bend left] \arrow[rrdd, "\pi"', bend right] \arrow[rrd, "F_{X/S}", dashed] &      &                        &  &             \\
		& & X^{(p)} \arrow[rr, "F'"] \arrow[d, "\pi'"'] &  & X \arrow[d, "\pi"] \\
		& & S \arrow[rr, "F"']            &  & S          
	\end{tikzcd} \] where the middle square is by definition a pullback square. In particular, if $F \colon S \to S$ is an isomorphism (e.g. $S$ is a perfect field), then $X^{(p)}$ and $X$ can be identified as abstract schemes.
	\item Given a complex $\cA^{\bullet}$ with values in some abelian category and $i \in \bZ$, we denote by $\cH^i(\cA^{\bullet})$ its $i$--th cohomology object.
	\item Given a variety $X$, we let $D_{\coh}(X)$ denote the derived category of complex of $\cO_X$--modules $\cA^{\bullet}$ such that $\cH^i(\cA^{\bullet})$ is coherent for all $i \in \bZ$. We let $D^b_{\coh}(X) \inc D_{\coh}(X)$ denote the full subcategory of complexes $\cA^{\bullet}$ such that $\cH^i(\cA^{\bullet}) = 0$ for $\abs{i} \gg 0$.
\end{itemize}

\section{Preliminaries}

Throughout, fix a positive integer $s > 0$.

\subsection{Cartier modules and crystals}

As stated in the introduction, the main actors in positive characteristic generic vanishing theory are \emph{Cartier modules} and \emph{crystals}. Here we gather the basic definitions of these objects. We will not need advanced results about these objects in the paper. The reader interested in further details about these objects is invited to read \cite{Blickle_Bockle_Cartier_modules_finiteness_results, Baudin_Duality_between_perverse_sheaves_and_Cartier_crystals}.

\begin{defn}
	Let $X$ be a Noetherian $\bF_p$--scheme.
	\begin{itemize}
		\item An \emph{$s$--Cartier module} is a pair ($\cM$, $\theta$) where $\cM$ is a coherent $\cO_X$--module and $\theta \colon F^s_*\cM \to \cM$ is a morphism. We call $\theta$ the \emph{structural morphism} of $\cM$. 
		
		\item A morphism of $s$--Cartier modules $h \colon (\cM_1, \theta_1) \to (\cM_2, \theta_2)$ is a morphism of underlying $\cO_X$--modules $h \colon \cM_1 \to \cM_2$ making the square
		
		\[ \begin{tikzcd}
			F^s_*\cM_1 \arrow[rr, "F^s_*h"] \arrow[d, "\theta_1"'] && F^s_*\cM_2 \arrow[d, "\theta_2"] \\
			\cM_1 \arrow[rr, "h"']                                     && \cM_2                               
		\end{tikzcd} \]
		commute. The category of $s$--Cartier modules is denoted $\Coh_X^{C^s}$.
	\end{itemize}	
\end{defn}

\begin{rem}
	Since it will not cause any confusion in this paper, we shall simply write ``Cartier module'' and omit the $s$ from the wording.
\end{rem}

Note that one can take the pushforward of a Cartier module: if $f \colon X \to Y$ is a proper morphism of Noetherian $\bF_p$--schemes and $(\cM, \theta)$ is a Cartier module on $X$, then $(f_*\cM, f_*\theta)$ defines a Cartier module on $Y$ (the properness assumption was only here to ensure that $f_*\cM$ is coherent).

\begin{example}\label{ex:Cartier_operator}
	Let us give the arguably most important example of a Cartier module (with $s = 1$): the sheaf of top differential forms together with the \emph{Cartier operator}. Assume that $X$ is a smooth variety, and let $\Omega_X^{\bullet}$ denote its de Rham complex. Note that although the differential $d$ is not $\cO_X$--linear, the maps $F_*d \colon F_*\Omega_X^j \to F_*\Omega_X^{j + 1}$ are $\cO_X$--linear (i.e. $d(f^ps) = f^pds$ for all $f \in \cO_X$ and $s \in \Omega_X^j$). \\
	
	The following important operator was defined by Cartier in \cite{Cartier_une_nouvelle_operation_sur_les_formes_differentielles}:
	\begin{equation}\label{Cartier_operator}
		\begin{tikzcd}[row sep=small]
			\Omega^j_X \arrow[rr, "C^{-1}"]            && \cH^j(F_*\Omega^{\bullet}_X)                               \\
			df_1 \wedge \dots \wedge df_j \arrow[rr, mapsto] && {[f_1^{p - 1}\dots f_j^{p -1}df_1 \wedge \dots \wedge df_j]},
		\end{tikzcd}
	\end{equation}
	for $f_1, \dots, f_j \in \cO_X$. In \emph{loc.cit}, the author shows that \autoref{Cartier_operator} is an isomorphism. Hence, we define a Cartier module structure on $\omega_X \coloneqq \Omega_X^{\dim X}$ by the composition 
	
	\[ \begin{tikzcd}
		F_*\omega_X \arrow[rr, two heads] &  & \cH^{\dim X}(F_*\Omega_X^{\bullet}) \arrow[rr, "(C^{-1})^{-1}"] &  & \omega_X.
	\end{tikzcd} \]

	More generally, if $X$ is a normal variety, we define $\omega_X \coloneqq j_*\omega_U$, where $U$ is the smooth locus of $X$. We then obtain a Cartier structure on $\omega_X$ in this case too.
\end{example}

\begin{rem}\label{rem:def_Cartier_mod}
	\begin{enumerate}
		\item\label{itm:abuse_iterates_structure_map} Let $(\cM, \theta)$ be a Cartier module. We will abuse notations as follows: \[ \theta^e \coloneqq \theta \circ F^s_*\theta \circ \dots \circ F^{(e-1)s}_*\theta \colon F^{es}_*\cM \to \cM. \]
		\item\label{itm:notation H^0_ss} Given a proper variety $X$ and a Cartier module $(\cM, \theta)$ on $X$, we set \[ H^i_{ss, \theta}(X, \cM) \coloneqq \bigcap_{e > 0} \im\left(H^i(X, F^{es}_*\cM) \xto{\theta^e} H^i(X, \cM)\right) \inc H^i(X, \cM). \] We also write $h^i_{ss, \theta}(X, \cM) \coloneqq \dim_k H^i_{ss, \theta}(X, \cM)$, 
		and  \[ \chi_{ss, \theta}(\cM) \coloneqq \sum_i (-1)^ih^i_{ss, \theta}(X, \cM). \] Whenever the context is clear, we shall omit the $\theta$ from this notation. In order to agree with the notations from the existing literature, whenever we work with the Cartier module $\omega_X$ from \autoref{Cartier_operator}, we set \[ S^0(X, \omega_X) \coloneqq H^0_{ss}(X, \omega_X). \]
	\end{enumerate}
\end{rem}

\begin{defn}
		A Cartier module $(\cM, \theta)$ is said to be \emph{nilpotent} if $\theta^e = 0$ for some $e > 0$. These objects form a Serre subcategory of $\Coh_X^{C^s}$ (meaning that being nilpotent is stable under subobjects, quotients and extensions). The corresponding quotient category (see \stacksproj{02MS}) is denoted by $\Crys_X^{C^s}$, and the objects of this quotient category are denoted \emph{Cartier crystals}.
\end{defn}

Intuitively, working in the category of crystals allows us to forget about nilpotent phenomena. For example, a Cartier module is the zero object as a Cartier crystal if and only if it is nilpotent. More generally, a morphism $f \colon \cM \to \cN$ of Cartier modules is an isomorphism at the level of Cartier crystals if and only if both $\ker(f)$ and $\coker(f)$ are nilpotent.

As it turns out, we will only briefly work with quotient categories, and from a formal perspective. Thus, knowing the precise definitions of these objects is not necessarily relevant apart from the intuition described above in order to understand this article.

\begin{notation}\label{notation_crystals}
	If $\cM_1$ and $\cM_2$ are two Cartier modules which are isomorphic as crystals, then we shall write $\cM_1 \sim_C \cM_2$.
\end{notation}

\subsection{$V$--modules and the Fourier--Mukai transform}

As first observed by Hacon in \cite{Hacon_A_derived_category_approach_to_generic_vanishing} and then later studied by Pareschi-Popa in \cite{Pareschi_Popa_GV_sheaves_FM_transform_and_generic_vanishing}, generic vanishing can be understood by how a given sheaf of interest on an abelian variety behaves with respect to the \emph{Fourier--Mukai transform}. It turns out that the image of Cartier modules under this operation are so--called \emph{$V$--modules}. 

Let us start by explaining what these objects are. Recall that throughout the paper, we fixed an abelian variety $A$. 

\begin{defn}[{\cite[Section 5]{Hacon_Pat_GV_Geom_Theta_Divs}}]
	The \emph{relative Verschiebung} $V_{A/k} \colon \bighat{A}^{(p)} = \bighat{A^{(p)}} \to \bighat{A}$ is by definition the dual isogeny of the relative Frobenius $F_{A/k} \colon A \to A^{(p)}$.
	
	Since $k$ is perfect, it follows by definition that $F' \colon \bighat{A}^{(p)} \to \bighat{A}$ is an isomorphism. The induced morphism $V \coloneqq V_{A/k} \circ (F')^{-1} \colon \bighat{A} \to \bighat{A}$ is called the \emph{absolute Verschiebung}. By construction, the diagram
	\[ \begin{tikzcd}
		\bighat{A} \arrow[rr, "V"] \arrow[d] &  & \bighat{A} \arrow[d] \\
		\Spec k \arrow[rr, "F^{-1}"]         &  & \Spec k             
	\end{tikzcd} \] commutes.
	\begin{itemize}
		\item A \emph{$V^s$--module} on $\bighat{A}$ is a pair $(\cN, \theta)$ where $\cN$ is a coherent sheaf on $\bighat{A}$ with a morphism $\theta \colon \cN \to V^{s, *}\cN$. A morphism of $V^s$--modules is defined as for Cartier modules: it is a morphism of underlying sheaves commuting with the $V^s$--module structures. The category of $V^s$--modules on $\bighat{A}$ is denoted $\Coh_{\bighat{A}}^{V^s}$.
		\item A $V^s$--module $(\cM, \theta)$ is nilpotent if $\theta^e = 0$ for some $e \geq 0$ (we abuse notations as in \autoref{rem:def_Cartier_mod}.\autoref{itm:abuse_iterates_structure_map}). Nilpotent modules form a Serre subcategory, whose quotient is the category of \emph{$V^s$--crystals}, denoted $\Crys_{\bighat{A}}^{V^s}$. If $\cN_1$ and $\cN_2$ are $V^s$--modules which are isomorphic as crystals, we shall write $\cN_1 \sim_V \cN_2$. 
		\item A $V^s$--module $(\cN, \theta)$ is said to be \emph{injective} if $\theta$ is injective.
	\end{itemize}
\end{defn}
\begin{rem}\label{canonical_injective_V_module}
	\begin{itemize}
		\item As for Cartier modules, we will omit the $s$ in the name ``$V^s$--module''.
		\item For any $V$--module $\cN$, there exists a canonical injective $V$--module $\cN_{\injj}$ with a surjective map $\cN \surj \cN_{\injj}$, inducing an isomorphism of crystals. Explicitly, $\cN_{\injj}$ is given by the image of $\cN \to \colim V^{es, *}\cN$.
		
		Note that the induced surjection $\cN \to \cN_{\injj}$ is indeed an isomorphism of $V$--crystals, since $\ker(f)$ is a nilpotent $V$--module by construction.
	\end{itemize}
\end{rem}

\begin{defn}
		Let $\phi$ denote the action of the Frobenius on $H^1(A, \cO_A)$. By \cite[Corollary p.143]{Mumford_Abelian_Varieties}, we can write \[ H^1(A, \cO_A) \cong H^1(A, \cO_A)_{ss} \oplus H^1(A, \cO_A)_{\nilp}, \] where $H^1(A, \cO_A)_{ss}$ has a basis of fixed points by $\phi$, and $H^1(A, \cO_A)_{\nilp}$ is nilpotent under the action of $\phi$. The $p$--rank of $A$ is then by definition \[ r(A) \coloneqq \dim_k H^1(A, \cO_A)_{ss} \in \{0, \dots, g\}. \]
\end{defn}

\begin{defn}
	\begin{itemize}
		\item We say that $A$ is \emph{ordinary} if $r(A) = g$.
		\item An elliptic curve $E$ is said to be \emph{supersingular} if $r(E) = 0$.
		\item We say that $A$ is \emph{supersingular} if it is isogenous to a product of supersingular elliptic curves (in particular, $r(A) = 0$, although this is not equivalent in general, see e.g. \cite[Table 4.3]{Pries_A_short_guide_to_p_torsion_of_AV_in_cgar_p}).
		\item By Poincaré's complete reducibility theorem (see e.g. \cite[Theorem 12.2]{Moonen_van_der_Geer_Abelian_varieties}), any abelian variety is isogeneous to a product of simple abelian varieties. We say that $A$ has \emph{no supersingular factor} if all these simple pieces are not supersingular. 
	\end{itemize}
\end{defn}

\begin{prop}\label{basic facts about AV}
	The following holds: 
	\begin{enumerate}
		\item\label{itm:ordinary iff Verschiebung etale} the absolute Verschiebung morphism $V$ is étale if and only if $A$ is ordinary;
		\item\label{itm:p_rank stable under isogenies} two isogeneous abelian varieties have the same $p$--rank;
		\item\label{itm:A_surjects_to_B_and_A_ord_implies_B_ord} if $f \colon A \to B$ is a surjective morphism of abelian varieties and $A$ is ordinary, then so is $B$;
		\item\label{itm:ord_no_ss_factor} an ordinary abelian variety has no supersingular factor.
	\end{enumerate}
\end{prop}
\begin{proof}
	It respectively follows from \cite[Proposition 2.3.2]{Hacon_Pat_GV_Characterization_Ordinary_AV}, \cite[p. 147]{Mumford_Abelian_Varieties}, \cite[Lemma 2.3.4]{Hacon_Pat_GV_Characterization_Ordinary_AV} and \cite[Lemma 2.3.5]{Hacon_Pat_GV_Characterization_Ordinary_AV}.
\end{proof}

We now recall the definition of the (symmetric) Fourier--Mukai transform (see \cite{Mukai_Fourier_Mukai_transform, Schnell_Fourier_Mukai_transform_made_easy}).
\begin{defn}
	\begin{itemize}
		\item The \emph{Fourier--Mukai transforms} are the functors $R\bighat{S} \colon D_{\coh}(A) \to D_{\coh}(\bighat{A})$ and $RS \colon D_{\coh}(\bighat{A}) \to D_{\coh}(A)$ defined by 
		\[ R\hat{S}(\cM^{\bullet}) \coloneqq Rp_{\bighat{A}, *}(Lp_A^*\cM^{\bullet} \otimes \cP)  \esp \mbox{ and } \esp RS(\cN^{\bullet}) \coloneqq Rp_{A, *}(Lp_{\bighat{A}}^*\cN^{\bullet} \otimes \cP), \] where $p_A \colon A \times \bighat{A} \to A$ and $p_{\bighat{A}} \colon A \times \bighat{A} \to \bighat{A}$ denote the projections.
		\item The \emph{symmetric Fourier--Mukai transforms} are the functors $\FM_A \coloneqq R\bighat{S} \circ D_A$ and $\FM_{\bighat{A}} \coloneqq RS \circ D_{\bighat{A}}$, where $D_A \coloneqq \cR\HHom(-, \omega_A[g])$ (and similarly for $D_{\bighat{A}}$).
	\end{itemize}
\end{defn}

\begin{rem}
	Although the precise definition of the (symmetric) Fourier--Mukai transforms may look technical, its properties are more important than the precise definition itself. In particular, the reader may accept \autoref{properties_Fourier_Mukai_transform} and \autoref{GV_Hacon_Pat} as black boxes which is sufficient for doing research in this field.
\end{rem}

The Fourier--Mukai transform has the following important properties (independently of the characteristic), which are central in generic vanishing theory. First recall that a vector bundle $\cV$ is \emph{unipotent} if it admits a filtration by sub--vector bundles whose graded pieces are isomorphic to $\cO_A$.

\begin{thm}\label{properties_Fourier_Mukai_transform}
	Let $\cM$ be a coherent sheaf on $A$.
	\begin{enumerate}
		\item\label{itm:support} For all $i \notin \{-g, \dots, 0\}$, then $\cH^i(\FM_A(\cM)) = 0$ for all $i \notin \{-g, \dots 0\}$. The same holds for $\FM_{\bighat{A}}$.
		\item\label{itm:equiv_cat} The functors $\FM_A$ and $\FM_{\bighat{A}}$ are inverses of each other, and hence equivalences of categories.
		\item\label{itm:equiv_cat_unipotent} We have $\FM_A(\omega_A) = k(0)$ (i.e. the skyscraper sheaf at $0 \in \bighat{A}$). In particular, the functor $\FM_A$ induces an equivalence of categories between unipotent vector bundles and coherent modules on $\bighat{A}$ supported at the point $0$.
		\item\label{itm:Fourier_Mukai_and_translation} For any $\alpha \in \bighat{A}$, \[ \FM_{\bighat{A}} \hspace{1 mm}\circ \esp T_\alpha^* \cong \cP_{-\alpha} \otimes \FM_{\bighat{A}},\] where $T_\alpha$ denotes the translation map by $\alpha$.
		\item\label{itm:behaviour_pushforwards_and_pullbacks} Let $f \colon A \to B$ be a morphism of abelian varieties, and let $\bighat{f} \colon \bighat{B} \to \bighat{A}$ denote the dual morphism. Then we have \[ \FM_B \: \circ \: Rf_* = L\bighat{f}^* \circ \FM_A. \]
		Here are two important special cases:
		\begin{itemize}
			\item if $\cM$ is a coherent sheaf on $A$, then for all $i \geq 0$, \[ H^i(A, \cM)^{\vee} \cong \Tor_i(\FM_A(\cM), k(0)). \]
			\item if $\cM$ is a Cartier module, then each $\cH^i(\FM_A(\cM))$ has an induced $V$--module structure.
		\end{itemize}
		\item\label{itm:more_gen_version} For any $\alpha \in \Pic^0(A)$, we have \[ H^i(A, \cM \otimes \cP_{\alpha})^{\vee} \cong \Tor_i(\FM_A(\cM), k(-\alpha)). \] 
		\item\label{itm:support_H0} We have that \[ \Supp \left(\cH^0\FM_A(\cM)\right) = \set{ \alpha \in \bighat{A}}{ H^0(A, \cM \otimes \cP_{-\alpha}) \neq 0}.\]
	\end{enumerate}
\end{thm}
\begin{proof}
	Part \autoref{itm:support} follows from the first lines in the proof of \cite[Theorem 3.1.1]{Hacon_Pat_GV_Characterization_Ordinary_AV}. Parts \autoref{itm:equiv_cat}, \autoref{itm:Fourier_Mukai_and_translation} and \autoref{itm:behaviour_pushforwards_and_pullbacks} follow respectively from \cite[Theorem 3.2, Propositions 4.1 and 5.1]{Schnell_Fourier_Mukai_transform_made_easy}. Part \autoref{itm:equiv_cat_unipotent} can be found in \cite[Example 2.9]{Mukai_Fourier_Mukai_transform} (it also follows from \cite[Corollary 1 p. 129]{Mumford_Abelian_Varieties}). Let us show the last two properties for the sake of the reader (even though they are well--known). We start with \autoref{itm:more_gen_version}, so let $\alpha \in \bighat{A}$. Then we have
	
	\begin{align*}
		H^i(A, \cM \otimes \cP_\alpha)^{\vee} & \expl{\cong}{\autoref{itm:behaviour_pushforwards_and_pullbacks}} \Tor_i(\FM_A(\cM \otimes \cP_\alpha), k(0)) \\
		& \expl{\cong}{\autoref{itm:Fourier_Mukai_and_translation}} \Tor_i(T_{\alpha}^*\FM_A(\cM), k(0)) \\
		& \cong \Tor_i(\FM_A(\cM), k(- \alpha)). \\
	\end{align*}

	Let us finally show \autoref{itm:support_H0}. Since $\FM_A(\cM)$ is supported in non--positive degrees by \autoref{itm:support}, it follows that \[ \Tor_0(\FM_A(\cM), k(\alpha)) \cong \Tor_0(\cH^0\FM_A(\cM), k(\alpha)) = \cH^0\FM_A(\cM) \otimes k(\alpha)\] for all $\alpha$. Hence, 
	\[ \Supp(\cH^0\FM_A(\cM)) \expl{=}{Nakayama's lemma} \set{\alpha \in \bighat{A}}{\cH^0\FM_A(\cM) \otimes k(\alpha) \neq 0} \expl{=}{\autoref{itm:more_gen_version}} \set{ \alpha \in \bighat{A}}{ H^0(A, \cM \otimes \cP_{-\alpha}) \neq 0}. \]
\end{proof}

In \autoref{section_first_vanishing_thm}, we will need a few more properties of this functor, which we will explain there.

\subsection{A few commutative algebra lemmas}
Here we collect a few results from commutative algebra which will be useful later on.

\begin{lem}\label{meaning of vanishing of Tor}
	Let $(R, \fm, k)$ be a Noetherian local ring, let $M \neq 0$ be a finitely generated $R$--module and let $i \geq 0$. Then \[ \Tor_i(M, k) = 0 \iff \projdim(M) < i \iff \depth(M) > \depth(R) - i.  \]
\end{lem}
\begin{proof}
	The first equivalence holds by \cite[Lemma 1.(ii) in Chapter 19]{Matsumura_Commutative_Ring_Theory}, and the second equivalence is exactly the Auslander--Buchsbaum's formula (see e.g. \cite[Theorem 19.1]{Matsumura_Commutative_Ring_Theory}).
\end{proof}

The following result is well--known to experts, but unable to find a reference, we prove it now.

\begin{cor}\label{locus of non-zero Tor is closed}
	Let $\cM$ be a coherent sheaf on a Noetherian scheme $X$. Then for all $i \geq 0$, the locus \[ \set{x \in X}{\Tor_i(\cM, k(x)) \neq 0} \] is closed.
\end{cor}
\begin{proof}
	When $i = 0$, we know by Nakayama's lemma that this locus is exactly the support of $\cM$, which is closed. Hence, assume that $i \geq 1$. By \autoref{meaning of vanishing of Tor}, we have to show that the locus \[ \set{x \in X}{\projdim(\cM_x) < i} \] is open. Suppose that $l \coloneqq \projdim(\cM_x) < i$, let $U$ be an affine open around $x$, and set $U = \Spec A$, and $\cM$ correspond to the $A$--module $M$ over $U$. Furthermore, let \[  0 \to K \to P_{l - 1} \to \dots  P_0 \to M \to 0\] be an exact sequence with each $P_i$ finitely generated projective and where $K$ is the kernel of $P_l \to P_{l - 1}$. Since $\cM_x$ has projective dimension $l$, $K$ is necessarily projective at $x$. Therefore $K$ will also be free in a neighbourhood $x$, so the projective dimension of $\cM$ is then smaller or equal to $l$ in a neighborhood of $x$.
\end{proof}
\begin{lem}\label{lemma Tor(V*) = 0 at x iff Tor = 0 at p(x)}
	Let $R \to S$ be a flat morphism of local rings with respective residue field $k_R$ and $k_S$, and let $M$ be a finitely generated $R$--module. Then \[ \Tor_i(M, k_R) = 0 \iff \Tor_i(M \otimes_R S, k_S) = 0. \]
\end{lem}
\begin{proof}
	By \autoref{meaning of vanishing of Tor}, it is enough to show that $\projdim(M) = \projdim(M \otimes_R S)$. We can then argue as in the proof of \autoref{locus of non-zero Tor is closed} to see that this follows from the fact that $M$ is projective if and only if $M \otimes_R S$ is projective, see \stacksproj{00O1}. \qedhere
	
%
%
\end{proof}

\begin{cor}\label{Tor(V*) = 0 at x iff Tor = 0 at p(x)}
	Let $\cM$ be a coherent sheaf on $\bighat{A}$. Then \[ \Tor_i(\cM, k(p(\alpha))) = 0 \iff \Tor_i(V^*\cM, k(\alpha)) = 0. \] 
\end{cor}
\begin{proof}
	Use \autoref{lemma Tor(V*) = 0 at x iff Tor = 0 at p(x)} with $V \colon \bighat{A} \to \bighat{A}$ at $\alpha$, together with the fact that set--theoretically, $V = p$ (indeed, if $\alpha \in \bighat{A} = \Pic^0(A)$, then by construction $V(\alpha) = F^*\alpha \cong \alpha^{\otimes p}$).
\end{proof}

\section{Generic vanishing}

\subsection{The main technical vanishing statement}\label{section_first_vanishing_thm}

Let us recall the definition given in the introduction of ``sheaves that satisfy generic vanishing'', namely \emph{GV--sheaves}.

\begin{defn}\label{def:GV_sheaf}
	Let $\cM$ be a coherent sheaf on $A$. We say that $\cM$ is a \emph{GV--sheaf} if for all $i \neq 0$, \[ \cH^i\FM_A(\cM) = 0. \]
\end{defn}

\begin{rem}
	As shown in \cite{Hacon_A_derived_category_approach_to_generic_vanishing} and \cite{Pareschi_Popa_GV_sheaves_FM_transform_and_generic_vanishing}, this definition is equivalent to the one given in the introduction, namely that the closed subsets \[ V^i(\cM) \coloneqq \set{\alpha \in \Pic^0(A)}{H^i(A, \cM \otimes \alpha) \neq 0} \inc \Pic^0(A) \] have codimension at least $i$. 
\end{rem}

In characteristic zero, fundamental theorems at the origin of generic vanishing theory is the fact that (higher) pushforwards of $\omega_X$ are GV--sheaves:

\begin{thm}[{\cite{Green_Lazarsfeld_Generic_vanishing, Hacon_A_derived_category_approach_to_generic_vanishing}, see also \cite[1.1.1]{Hacon_Pat_GV_Geom_Theta_Divs}}]\label{GV_char_zero}
	Let $a \colon X \to A$ be any morphism of complex varieties, where $X$ is smooth projective and $A$ is an abelian variety. Then each higher pushforward $R^ja_*\omega_X$ is a GV--sheaf.
\end{thm}

Unfortunately, this is known to fail in positive characteristic (\cite{Filipazzi_GV_fails_in_pos_char}). Hacon and Patakfalvi found the following replacement of this vanishing on an abelian variety, stating roughly that Cartier modules are ``GV--sheaves up to nilpotence''.

\begin{thm}[{\cite[Theorem 1.3.1]{Hacon_Pat_GV_Characterization_Ordinary_AV}}]\label{GV_Hacon_Pat}
	Let $\cM$ be a Cartier module on $A$. Then for all $i \neq 0$, the $V$--module \[ \cH^i\FM_A(\cM) \] is nilpotent.
\end{thm}

The main objective of this subsection is to give a proof of this result. Unlike later, our proof will essentially be the same as that in \cite{Hacon_Pat_GV_Characterization_Ordinary_AV}, and we really only give it for the sake of completeness. \\

Before going to the proof, let us comment on the theorem above. Even though \autoref{GV_Hacon_Pat} is mostly applied to the Cartier module $a_*\omega_X$, it is important to point out that this theorem applies to \emph{any} Cartier module. This is very useful in practice.

One may be surprised, as this seems to be much more general than its characteristic zero counterpart. For example, we can put a Cartier structure on any coherent sheaf: the trivial one. But then, note that \autoref{GV_Hacon_Pat} becomes a trivial statement, and it is of course useless in applications. A natural follow-up question is then the following: when does a coherent sheaf on an abelian variety admit non--trivial Cartier module structure? It turns out that this itself is some sort of positivity property. For example, if $L$ is a line bundle on $A$ and $F^s_*L \to L$ is a non--zero morphism, then $L$ is in fact a GV--sheaf. Indeed, the same computation as in \cite[Beginning of Page 12]{Blickle_Schwede_p-1_linear_maps_in_algebra_and_geometry} shows that $(L \otimes \omega_A^{\vee})^{\otimes (p^s - 1)} \cong L^{\otimes (p^s - 1)}$ is effective. We can then apply the following result.

\begin{lemma}\label{example_when_line_bundles_are_GV}
	Let $L = \cO_A(D)$ be a line bundle on $A$. Then $L$ is a GV--sheaf if and only if $D$ is numerically equivalent to an effective $\bQ$--divisor.
\end{lemma} 
\begin{proof}
	Assume first that $D$ is numerically equivalent to an effective $\bQ$--divisor, and let us show that $L$ is a GV--sheaf. By assumption, there exists $n > 0$ and $\alpha \in \Pic^0(A)$ such that $H^0(A, L^{\otimes n} \otimes \alpha) \neq 0$. Pick $\beta \in \Pic^0(A)$ such that $\beta^{\otimes n} \cong \alpha$ (see \cite[Application 2 p.62]{Mumford_Abelian_Varieties}). By \autoref{properties_Fourier_Mukai_transform}.\autoref{itm:Fourier_Mukai_and_translation}, it is enough to show that the $\bQ$--effective line bundle $L_0 \coloneqq L \otimes \beta$ is a GV--sheaf.
	
	Note that for any $m \geq 1$, it is enough to show that $m^*L_0$ is a GV--sheaf (this follows from \autoref{properties_Fourier_Mukai_transform}.\autoref{itm:behaviour_pushforwards_and_pullbacks}). Since $L_0$ and $L_0 \otimes (-1)^*L_0$ are $\bQ$--effective (and hence semiample by the argument of \cite[Application p. 60, (i) $\implies$ (iii)]{Mumford_Abelian_Varieties}), there exists $m \gg 0$ such that \[ m^*L_0 \expl{\cong}{\cite[{Corollary 3 p. 59}]{Mumford_Abelian_Varieties}} L_0^{\otimes m} \otimes \left(L_0 \otimes (-1)^*L_0\right)^{\otimes \frac{m(m - 1)}{2}} \] is globally generated. We then reduced to the case where $L_0$ is globally generated.
	
	In this case, we know by \cite[Discussion after Lemma 2.21]{Moonen_van_der_Geer_Abelian_varieties} (see also Section 8.2 in \emph{loc.cit.}) that the Stein factorization of the morphism given by $L_0$ maps to an abelian variety, so we can write $L_0 \cong \pi^*M$, where $\pi \colon A \to B$ is a fibration to an abelian variety $B$, and $M$ is ample on $B$. An ample line bundle is automatically a GV--sheaf by Kodaira vanishing (which holds on any abelian variety by \cite[Application p.60, Thm. p.150 and Cor. p. 159]{Mumford_Abelian_Varieties}), so $\pi^*M$ is also a GV--sheaf by \autoref{properties_Fourier_Mukai_transform}.\autoref{itm:behaviour_pushforwards_and_pullbacks}. 
	
	We finish with proving the converse direction, so assume that $L$ is a GV--sheaf. Then by assumption and \autoref{properties_Fourier_Mukai_transform}.\autoref{itm:equiv_cat}, we know that $\cH^0\FM_A(L) \neq 0$, so its support is non--empty. We then obtain by \autoref{properties_Fourier_Mukai_transform}.\autoref{itm:gv_link_with_section} that $L \otimes \alpha$ is effective for some $\alpha \in \Pic^0(A)$, thereby concluding the proof.
\end{proof}

\begin{warning}
	The situation is more complicated for higher rank vector bundles. For example, there exists an abelian surface $A$ and a vector bundle $\cM$ admitting a Cartier module structure $F_*\cM \to \cM$ which is surjective, and which is not a GV--sheaf. Namely, it is given by $a_*\omega_S$, where $a \colon S \to A$ is the morphism given by the main result of \cite{Filipazzi_GV_fails_in_pos_char}.
\end{warning}

Let us now get to the proof of \autoref{GV_Hacon_Pat}. Again, \emph{we stress that the proof given here is the same the original one}, and we only include it for the sake of completeness. Let us first state more properties of the Fourier--Mukai transform, which will only be used in this section.

\begin{notation}
	Given an ample line bundle $L$ on $\bighat{A}$, the isogeny $\phi_L \colon \bighat{A} \to A$ is defined by sending $\alpha \in \bighat{A}$ to the line bundle $T_\alpha^*L \otimes L^{-1} \in \Pic^0(\bighat{A}) = A$.
	
	We also set $\hat{L} \coloneqq RS(L)$. An immediate application of Kodaira vanishing and cohomology and base change shows that $RS(L) = R^0S(L)$.
\end{notation}

\begin{lemma}\label{further_props_FM}
	The following properties hold:
	\begin{enumerate}
		\item\label{itm:phi_oullback_easy} $\phi_L^*(\hat{L}^{\vee}) \cong L^{\oplus h^0(\bighat{A}, L)}$;
		\item\label{itm:cohom_FM} for all $i \in \bZ$ and $\cM \in \Coh_A$, we have an isomorphism $H^{-i}(\bighat{A}, \FM_A(\cM) \otimes L) \cong H^i(A, \cM \otimes \hat{L}^{\vee})^{\vee}$.
	\end{enumerate}
\end{lemma}
\begin{proof}
	The first statement is a special case of \cite[Proposition 3.11]{Mukai_Fourier_Mukai_transform}. To show the second statement, note that
	\begin{align*} 
		\RGamma(\bighat{A}, \FM_A(\cM) \otimes L) & \cong \RGamma(\bighat{A}, Rp_{\bighat{A}, *}(p_A^*D_A(\cM) \otimes \cP) \otimes L) \\
		& \expl{\cong}{projection formula} \RGamma(\bighat{A}, Rp_{\bighat{A}, *}(p_A^*D_A(\cM) \otimes \cP \otimes p_{\bighat{A}}^*L)) \\
		& \cong \RGamma(A \times \bighat{A}, p_A^*D_A(\cM) \otimes \cP \otimes p_{\bighat{A}}^*L) \\
		& \cong \RGamma(A, Rp_{A, *}(p_A^*D_A(\cM) \otimes \cP \otimes p_{\bighat{A}}^*L)) \\
		& \expl{\cong}{projection formula and definition of $RS$} \RGamma(A, D_A(\cM) \otimes \hat{L}) \\
		& \cong \RGamma(A, D_A(\cM \otimes \hat{L}^{\vee})) \\
		& \expl{\cong}{Grothendieck duality, see e.g. \stacksproj{0A9Q} with $f \colon A \to \Spec k$} \RGamma(A, \cM \otimes \hat{L}^{\vee})^{\vee}.
	\end{align*}
	so we conclude the proof by taking $\cH^{-i}$ on both sides.
\end{proof}

\begin{defn}
	We say that an inverse system of coherent sheaves $\{\cM_e\}_{e \geq 1}$ on $A$ satisfies $(*)$ if the following holds: for any ample line bundle $L$ on $\bighat{A}$, there exists $e \geq 0$ such that for all $f \geq e$ and $i > 0$, we have the vanishing \[ H^i(A, \cM_f \otimes \hat{L}^{\vee}) = 0. \]
\end{defn}

\begin{lemma}\label{prop_star_implies_GV_system}
	Let $\{\cM_e\}_{e \geq 0}$ be an inverse system of coherent sheaves on $A$ satisfying property $(*)$. Then for all $j \neq 0$, \[ \colim_e \cH^j(\FM_A(\cM_e)) = 0. \] In other words, for all $e \geq 0$, there exists $f \geq e$ such that the morphism $\cH^j(\FM_A(\cM_e)) \to$ $ \cH^j(\FM_A(\cM_f))$ is zero for all $j \neq 0$.
\end{lemma}
\begin{proof}
	For any $e \geq 0$, write $\cN_e^{\bullet} \coloneqq \FM_A(\cM_e)$. The vanishing is immediate for $j > 0$ by \autoref{properties_Fourier_Mukai_transform}.\autoref{itm:support}. Assume by contradiction that the result does not hold, and let $j < 0$ be the smallest integer such that $\colim \cH^j(\cN_e^{\bullet}) \neq 0$. Fix an integer $e_0 > 0$ such that the natural map $\cH^j(\cN_{e_0}^{\bullet}) \to \colim \cH^j(\cN_e^{\bullet})$ is non--zero, and let $L$ be an ample line bundle such that $\cH^j(\cN_{e_0}^{\bullet}) \otimes L$ is globally generated. By hypothesis, the morphism \[ H^0\left(\bighat{A}, \cH^j(\cN_{e_0}^{\bullet}) \otimes L\right) \longrightarrow H^0\left(\bighat{A}, \colim_e \cH^j(\cN_e^{\bullet}) \otimes L\right)\] is then non--zero. Let us show that this is impossible, by showing that the right--hand side vanishes. We have that
	\begin{align*}
		H^0(\bighat{A}, \colim \cH^j(\cN_e^{\bullet}) \otimes L) & \expl{\cong}{$j$ is minimal, so the complex $\colim \cN_e^{\bullet} \otimes L$ lives in degrees $\geq j$} H^j(\bighat{A}, \colim \cN_e^{\bullet} \otimes L) \\
		& = \colim H^j(\bighat{A}, \cN_e^{\bullet} \otimes L) \\
		& \expl{\cong}{\autoref{further_props_FM}.\autoref{itm:cohom_FM}} \colim H^{-j}(A, \cM_e \otimes \hat{L}^{\vee})^{\vee}.
	\end{align*}
	We deduce by property $(*)$ that this group vanishes. \qedhere
	
	%

\end{proof}

\begin{lemma}\label{Cartier_mod_gives_GV_system}
	Let $(\cM, \theta)$ be a Cartier module on $A$. Then the inverse system \[ \{\theta^e \colon F^{e(s + 1)}_*\cM \to F^{es}_*\cM\} \] satisfies $(*)$.
\end{lemma}
\begin{proof}
	Fix $L$ ample on $A$ and $e \gg 0$. Let us show that for all $i > 0$ and all $\alpha \in \Pic^0(A)$, we have $H^i(A, F^{es}_*\cM \otimes \hat{L}^{\vee} \otimes \alpha) = 0$ (the case $P \cong \cO_A$ will conclude the proof). By the projection formula, is it enough to show that $H^i(A, \cM \otimes F^{es, *}\hat{L}^{\vee} \otimes \alpha) = 0$ for all $i > 0$ and all $\alpha \in \Pic^0(A)$. By cohomology and base change, this amounts to showing that the complex \[ R\hat{S}(\cM \otimes F^{es, *}(\hat{L}^{\vee}))\] is in fact a sheaf in degree zero. This is equivalent to showing that \[ \hat{\varphi}_{L, *}R\hat{S}(\cM \otimes F^{es, *}(\hat{L}^{\vee})) \] is a sheaf in degree zero. However, we have an isomorphism \[ \hat{\varphi}_{L, *}R\hat{S}(\cM \otimes F^{es, *}(\hat{L}^{\vee})) \cong R\hat{S}(\phi_L^*(\cM \otimes F^{es, *}(\hat{L}^{\vee}))), \] by \cite[3.4]{Mukai_Fourier_Mukai_transform}, so again by cohomology and base change, we have to show that \[ H^i(A, \phi_L^*(\cM \otimes F^{es, *}\hat{L}^{\vee}) \otimes \alpha)  = 0 \] for all $i > 0$ and $\alpha \in \Pic^0(A)$ (note that one given $e > 0$ has to give this vanishing for all $\alpha \in \Pic^0(A)$ at the same time). Since 
	\begin{align*} 
		\phi_L^*(\cM \otimes F^{es, *}\hat{L}^{\vee}) \otimes \alpha & \cong \phi_L^*\cM \otimes F^{es, *}\phi_L^*\hat{L}^{\vee} \otimes \alpha \\
		& \expl{\cong}{\autoref{further_props_FM}.\autoref{itm:phi_oullback_easy}} \left(\phi_L^*\cM \otimes F^{es, *}L \otimes \alpha \right)^{\oplus h^0(\bighat{A}, L)}  \\
		& \cong \left(\phi_L^*\cM \otimes L^{\otimes p^{es}} \otimes \alpha \right)^{\oplus h^0(\bighat{A}, L)},
	\end{align*}  
	we conclude the proof by Fujita vanishing (\cite{Fujita_Vanishing_theorem_for_semipositive_line_bundles}).
\end{proof}

\begin{proof}[Proof of \autoref{GV_Hacon_Pat}]
	This follows from \autoref{Cartier_mod_gives_GV_system} and \autoref{prop_star_implies_GV_system}.
\end{proof}

\subsection{Equivalence between Cartier crystals and $V$--crystals on dual abelian varieties}\label{sec:equiv_cat}

Let us deduce formally from \autoref{GV_Hacon_Pat} the following result, which implicitly contains the statements of \cite[Theorem 5.2 and Corollary 5.3]{Hacon_Pat_GV_Geom_Theta_Divs}.

\begin{thm}\label{equivalence_V_crystals_and_Cartier_crystals}
	The functors \[ \begin{tikzcd}[row sep=0.3em]
		(\Coh_A^{C^s})^{op} \arrow[rr] &  & \Coh_{\bighat{A}}^{V^s} \\
		\cM \arrow[rr, maps to]        &  & \cH^0(\FM_A(\cM))      
	\end{tikzcd} \] and 
	\[ \begin{tikzcd}[row sep=0.3em]
		\Coh_{\bighat{A}}^{V^s} \arrow[rr] &  & (\Coh_A^{C^s})^{op} \\
		\cN \arrow[rr, maps to]        &  & \cH^0(\FM_A(\cN))      
	\end{tikzcd} \] factor through functors on crystals $(\Crys_A^{C^s})^{op} \to \Crys_{\bighat{A}}^{V^s}$ and $\Crys_{\bighat{A}}^{V^s} \to (\Crys_A^{C^s})^{op}$, and they become inverses of one another. In particular, we have an equivalence of categories \[ (\Crys_A^{C^s})^{op} \cong \Crys_{\bighat{A}}^{V^s}. \]
\end{thm}

The proof will require a few intermediate results. The main intermediate step is to show that the analogue of \autoref{GV_Hacon_Pat} holds for $V$--modules, see \autoref{vanishing_for_V_modules} (this will be a formal consequence of \autoref{GV_Hacon_Pat} and properties of the Fourier--Mukai transform). 

Recall the definitions of the truncation functors: given a complex $\cA^{\bullet}$ and $i \in \bZ$, we define its truncations \[ (\tau_{\leq i}\cA^{\bullet})^p \coloneqq \begin{cases}
	\cA^p & \mbox{ if } p < i; \\
	\ker(d^i) & \mbox{ if } p = i; \\
	0 & \mbox{ if }p > i,
\end{cases} \] and \[ (\tau_{\geq i}\cA^{\bullet})^p \coloneqq \begin{cases}
	0 & \mbox{ if } p < i; \\
	\coker(d^{i - 1}) & \mbox{ if } p = i; \\
	\cA^p & \mbox{ if }p > i.
\end{cases} \]
By \stacksproj{08J5}, we then have exact triangles \[ \begin{tikzcd}
	\tau_{\leq i}\cA^{\bullet} \arrow[rr] &  & \cA^{\bullet} \arrow[rr] &  & \tau_{\geq i + 1}\cA^{\bullet} \arrow[rr, "{+1}"] &  & {}
\end{tikzcd} \]

We will repeatedly use the following standard fact about derived categories: given a diagram \[ \begin{tikzcd}
	&  & \cF^{\bullet} \arrow[d, "f"]  &  &                        &  &    \\
	\cA^{\bullet} \arrow[rr, "a"] &  & \cB^{\bullet} \arrow[rr, "b"] &  & \cC^{\bullet} \arrow[rr, "{+1}"] &  & {}
\end{tikzcd} \] in a derived category where the bottom arrows form an exact triangle, we have that if $bf = 0$, then $f$ factors through $a$.

\begin{lemma}\label{extension_of_derived_nilp_is_nilp}
	Consider a commutative diagram of exact triangles in $D^b_{\coh}(A)$:
	
	\[ \begin{tikzcd}
		F^s_*\cM_1^{\bullet} \arrow[rr, "F^s_*a"] \arrow[d, "\theta_1"] &  & F^s_*\cM_2^{\bullet} \arrow[rr, "F^s_*b"] \arrow[d, "\theta_2"] &  & F^s_*\cM_3^{\bullet} \arrow[rr, "{+1}"] \arrow[d, "\theta_3"] &  & {} \\
		\cM_1^{\bullet} \arrow[rr, "a"]                                 &  & \cM_2^{\bullet} \arrow[rr, "b"]                                 &  & \cM_3^{\bullet} \arrow[rr, "{+1}"]                            &  & {}
	\end{tikzcd} \]
	If $\theta_1$ and $\theta_3$ are nilpotent, then so is $\theta_2$.
\end{lemma}
\begin{proof}
	Let $e > 0$ be such that $\theta_1^e = 0 =\theta_3^e$. Then $0 = \theta_3^eF^{es}_*b = b\theta_2^e$, so $\theta_2^e$ factors through $\cM_1^{\bullet}$, say via a map $c$. But then, $\theta_2^{2e} = a\theta_1^eF^e_*c = 0$.
\end{proof}

\begin{cor}\label{nilp_at_each_hi_is_nilp}
	Let $\cM^{\bullet} \in D^b_{\coh}(A)$, and let $\theta \colon F^s_*\cM^{\bullet} \to \cM^{\bullet}$ be a morphism. Assume that for all $i \in \bZ$, the Cartier module $\cH^i(\cM^{\bullet})$ is nilpotent. Then for some $e > 0$, the composition $\theta^e \colon F^{es}_*\cM^{\bullet} \to \cM^{\bullet}$ is zero.
\end{cor} 
\begin{rem}\label{nilp_at_each_hi_is_nilp_V_module}
	\begin{itemize}
		\item The same statement holds in the context of $V$--modules, and the proof is identical.
		\item One has to be careful when working with morphisms in the derived category. It is not true in general that if $f \colon \cF^{\bullet} \to \cG^{\bullet}$ is a morphism in a derived category such that $\cH^i(f) = 0$ for all $i \in \bZ$, then $f = 0$. For example, if $X$ is any smooth proper variety of dimension $d > 0$, then \[ 0 \neq H^d(X, \omega_X) = \Ext^d(\cO_X, \omega_X) \expl{=}{\cite[Theorem 10.7.4]{Weibel_Intro_to_homological_algebra}} \Hom_{D(\cO_X)}(\cO_X, \omega_X[d]), \] but any morphism $f \colon \cO_X \to \omega_X[d]$ must satisfy that $\cH^i(f) = 0$ for all $i \in \bZ$.
	\end{itemize}

\end{rem}
\begin{proof}
	If $\cM^{\bullet} = 0$, this is immediate. If not, let $i \in \bZ$ be the minimal integer such that $\cH^i(\cM^{\bullet}) \neq 0$. Then we have an exact triangle \[ \begin{tikzcd}
		\cH^i(\cM^{\bullet})[-i] \arrow[rr] &  & \cM^{\bullet} \arrow[rr] &  & \tau_{\geq i + 1}\cM^{\bullet} \arrow[rr, "+1"] &  & {}
	\end{tikzcd} \] so by \autoref{extension_of_derived_nilp_is_nilp} it is enough to show that $\tau_{\geq i + 1}\cM^{\bullet}$ is nilpotent. We then can keep going inductively and deduce the result.
\end{proof}

\begin{prop}[{\cite[Theorem 5.2]{Hacon_Pat_GV_Geom_Theta_Divs}}]\label{vanishing_for_V_modules}
	Let $\cN$ be a $V$--module. Then for all $i \neq 0$, the Cartier module $\cH^i\FM_{\bighat{A}}(\cN)$ is nilpotent.
\end{prop}
\begin{proof}
	Let $\cM^{\bullet} \coloneqq \FM_{\bighat{A}}(\cN)$ (this complex is supported in degrees $\leq 0$ by \autoref{properties_Fourier_Mukai_transform}.\autoref{itm:support}). We will show that the morphism \[ F^{s}_*\tau_{\leq -1}\cM^{\bullet} \longrightarrow \tau_{\leq -1}\cM^{\bullet} \] is nilpotent (this will conclude the proof by applying cohomology sheaves). There is a commutative diagram of exact triangles \[ \begin{tikzcd}
		F^{s}_*(\tau_{\leq -1}\cM^{\bullet}) \arrow[d] \arrow[rr] &  & F^{s}_*\cM^{\bullet} \arrow[d] \arrow[rr] &  & F^{s}_*\cH^0(\cM^{\bullet}) \arrow[d] \arrow[rr, "+1"] &  & {} \\
		\tau_{\leq -1}\cM^{\bullet} \arrow[rr]                     &  & \cM^{\bullet} \arrow[rr]                   &  & \cH^0(\cM^{\bullet}) \arrow[rr, "+1"]                   &  & {}
	\end{tikzcd} \]
	Applying $\FM_A$ and using points \autoref{itm:equiv_cat} and \autoref{itm:behaviour_pushforwards_and_pullbacks} of \autoref{properties_Fourier_Mukai_transform} gives an other commutative diagram of exact triangles \begin{equation}\label{diag:proof_easy_version} \begin{tikzcd}
		\FM_A(\cH^0(\cM^{\bullet})) \arrow[r] \arrow[d]  &  \cN \arrow[r] \arrow[d]   & \FM_A(\tau_{\leq- 1}\cM^{\bullet}) \arrow[r, "+1"] \arrow[d]   & {} \\
		{V^{s, *}\FM_A(\cH^0(\cM^{\bullet}))} \arrow[r]  & {V^{s, *}\cN} \arrow[r]  & {V^{s, *}\FM_A(\tau_{\leq- 1}\cM^{\bullet})} \arrow[r, "+1"]  & {}
	\end{tikzcd} \end{equation}
%
	
	\noindent Note that since $\cN$ is supported in degree zero and $\FM_A(\cH^0(\cM^{\bullet}))$ is supported in degrees $\leq 0$ by \autoref{properties_Fourier_Mukai_transform}.\autoref{itm:support}), we deduce that \[ \cH^j(\FM_A(\tau_{\leq- 1}\cM^{\bullet}))) = 0 \] for all $j \geq 1$ by the induced long exact sequence in cohomology.
	
	Consider the hypercohomology spectral sequence of $V$--modules induced by \autoref{diag:proof_easy_version}:  \[ \cH^a\FM_A(\cH^b(\tau_{\leq -1}(\cM^{\bullet}))) \implies \cH^{a - b}\FM_A(\tau_{\leq -1}(\cM^{\bullet})). \] By \autoref{GV_Hacon_Pat}, this spectral sequence degenerates at the level of $V$--\emph{crystals} and gives us that \[ \cH^0(\FM_A(\cH^b(\tau_{\leq -1}(\cM^{\bullet}))) \sim_V \cH^{-b}\FM_A(\tau_{\leq -1}(\cM^{\bullet})) \] for all $b \in \bZ$. If $b \geq 0$, then the module on the left is zero and if $b \leq -1$, then the module on the right is zero by the previous paragraph.
	
	We then have proven that for all $b \in \bZ$, the $V$--module $\cH^0(\FM_A(\cH^b(\tau_{\leq -1}(\cM^{\bullet}))))$ is nilpotent. Since this is also the case for each $\cH^i(\FM_A(\cH^b(\tau_{\leq -1}(\cM^{\bullet}))))$ with $i \neq 0$ by \autoref{GV_Hacon_Pat}, we deduce by \autoref{nilp_at_each_hi_is_nilp_V_module} that the map $\FM_A(\cH^b(\cM^{\bullet})) \to V^{s, *}\FM_A(\cH^b(\cM^{\bullet}))$ is nilpotent. By \autoref{properties_Fourier_Mukai_transform}.\autoref{itm:equiv_cat}, this is equivalent to the fact that each Cartier module $\cH^b(\tau_{\leq -1}(\cM^{\bullet}))$ is nilpotent, so the morphism \[ F^s_*\tau_{\leq -1}(\cM^{\bullet}) \to \tau_{\leq -1}(\cM^{\bullet}) \] is also nilpotent by \autoref{nilp_at_each_hi_is_nilp}.
\end{proof}

We now have all the tools to prove that Cartier crystals and $V$--crystals are equivalent.

%
%

\begin{proof}[Proof of \autoref{equivalence_V_crystals_and_Cartier_crystals}]
	By \autoref{properties_Fourier_Mukai_transform}.\autoref{itm:behaviour_pushforwards_and_pullbacks}, the functors $\cH^0\FM$ send Cartier modules to $V$--modules and vice--versa, and it follows from the definitions that they preserve nilpotence. Let us first show that they induce functors at the level of crystals, so consider the composition \[ \begin{tikzcd}
		(\Coh_A^{C^s})^{op} \arrow[rr, "\cH^0\FM_A"] &  & \Coh_{\bighat{A}}^{V^s} \arrow[rr] &  & \Crys_{\bighat{A}}^{V^s}.
	\end{tikzcd} \]
	In order to be able to apply \stacksproj{02MS} and factor the composition above through $(\Crys_A^{C^s})^{op}$, it is then enough to show that it is exact. Hence, let \[ \begin{tikzcd}
		0 \arrow[rr] &  & \cM_1 \arrow[rr] &  & \cM_2 \arrow[rr] &  & \cM_3 \arrow[rr] &  & 0
	\end{tikzcd} \] be an exact sequence of Cartier modules.  Applying $\FM_A$ gives an exact triangle 
	\[ \begin{tikzcd}
		\FM_A(\cM_3) \arrow[rr] &  & \FM_A(\cM_2) \arrow[rr] &  & \FM_A(\cM_1) \arrow[rr, "+1"] &  & {}
	\end{tikzcd} \] By \autoref{GV_Hacon_Pat} and the induced long exact sequence in cohomology sheaves, we deduce that the sequence  \[ \begin{tikzcd}
	0 \arrow[r]  & \cH^0(\FM_A(\cM_3)) \arrow[r]  & \cH^0(\FM_A(\cM_2)) \arrow[r]   & \cH^0(\FM_A(\cM_1)) \arrow[r]  & 0
\end{tikzcd} \] is exact up to nilpotence (i.e.	exact in $\Crys_{\bighat{A}}^{V^s}$). We can then apply \stacksproj{02MS} and deduce the existence of the factorization 	\[ \begin{tikzcd}
	(\Coh_A^{C^s})^{op} \arrow[rr, "\cH^0\FM_A"] \arrow[d] &  & \Coh_{\bighat{A}}^{V^s} \arrow[d] \\
	(\Crys_A^{C^s})^{op} \arrow[rr]                        &  & \Crys_{\bighat{A}}^{V^s}.        
	\end{tikzcd} \] The proof in the other direction is identical, replacing the use of \autoref{GV_Hacon_Pat} by that of \autoref{vanishing_for_V_modules}. \\
	
	Now, let $\cN$ be a $V$--module, and let $\cM_0 \coloneqq \cH^0\FM_{\bighat{A}}(\cN)$. By \autoref{properties_Fourier_Mukai_transform}.\autoref{itm:support}, there is an exact triangle 
	\[ \begin{tikzcd}
		\tau_{\leq -1}\FM_{\bighat{A}}(\cN) \arrow[rr] &  & \FM_{\bighat{A}}(\cN) \arrow[rr] &  & \cM_0 \arrow[rr, "+1"] &  & {}.
	\end{tikzcd} \]

	By \autoref{vanishing_for_V_modules} and \autoref{extension_of_derived_nilp_is_nilp}, there exists $e > 0$ such that the morphism $F^{es}_*\tau_{\leq -1}\FM_{\bighat{A}}(\cN) \to \tau_{\leq -1}\FM_{\bighat{A}}(\cN)$ is zero. We can apply $\FM_A$ to the above triangle and use \autoref{properties_Fourier_Mukai_transform}.\autoref{itm:equiv_cat} to obtain \[ \begin{tikzcd}
		\FM_A(\cM_0) \arrow[rr] &  & \cN \arrow[rr] &  & \FM_A\left(\tau_{\leq -1}\FM_{\bighat{A}}(\cN)\right) \arrow[rr, "+1"] &  & {}
	\end{tikzcd} \]
	Set $\ttilde{\cN} \coloneqq \FM_A\left(\tau_{\leq -1}\FM_{\bighat{A}}(\cN)\right)$. Since the natural map $\ttilde{\cN} \to V^{es, *}\ttilde{\cN}$ is zero, we deduce that for all $j \in \bZ$, the $V$--module $\cH^j(\ttilde{\cN})$ is nilpotent. By the exact triangle above, we have a long exact sequence of $V$--modules \[ \begin{tikzcd}
		\dots \arrow[r] & \cH^{-1}(\ttilde{\cN}) \arrow[r] &  \cH^0\FM_A(\cM_0) \arrow[r] & \cN \arrow[r] & \cH^0(\ttilde{\cN}) \arrow[r] & \dots
	\end{tikzcd} \]
	Thus, the kernel and cokernel of $\cH^0\FM_A(\cM_0) \to \cM$ are nilpotent $V$--modules, so this morphism is by definition an isomorphism as $V$--crystals. In other words, we found a natural isomorphism \[ (\cH^0\FM_A \circ \: \cH^0\FM_{\bighat{A}})(\cN) \to \cN \] of $V$--crystals, which is exactly what we needed to show. 
	
	The proof that there is a natural isomorphism $(\cH^0\FM_{\bighat{A}} \circ \: \cH^0\FM_A)(\cM) \to \cM$ on Cartier modules $\cM$ is identical.
\end{proof}

\subsection{Cohomological support loci}

Classically, generic vanishing came as the study of some cohomology vanishing loci. We will now move towards this direction. The main purposes for this section are the following:

\begin{enumerate}
	\item Slightly improve \cite[Theorem 1.1]{Hacon_Pat_GV_Geom_Theta_Divs}. As a new application of this result, we proved in \cite{Baudin_Duality_between_perverse_sheaves_and_Cartier_crystals} a generic vanishing results about perverse $\bF_p$--sheaves on abelian varieties.
	\item Simplify the proofs, using a more commutative algebraic point of view.
	\item Give examples to illustrate the theory.
\end{enumerate}

\noindent In characteristic zero, the main actors are the following closed subsets:

\begin{defn}
	Let $\cM$ be a coherent sheaf on $A$. We set \[ V^i(\cM) = \set{\alpha \in \bighat{A}}{H^i(A, \cM \otimes \cP_\alpha) \neq 0} \inc \bighat{A}. \]
\end{defn}

Recall that if $\cM$ is a GV--sheaf as in \autoref{def:GV_sheaf}, then its cohomology support loci $V^i(\cM)$ are of codimension at least $i$. We want to find an analoguous statement in the context of Cartier modules, so we need a definition of ``cohomology support loci'' in this setup. The definition that arose in \cite{Hacon_Pat_GV_Geom_Theta_Divs} is the following:

\begin{defn}
	Given a Cartier module $\cM$ on $A$, we set \[ W^i_F(\cM)  = \set{\alpha \in \bighat{A}}{\varprojlim H^i(A, F^{es}_*\cM \otimes \cP_\alpha) \neq 0} \subseteq \bighat{A}. \]
\end{defn}

Note that by the projection formula, \[ (p^s)^{-1}W^i_F(\cM) = W^i_F(\cM).\] 

\begin{rem}
	Whenever the context is clear, we will only write $V^i$ and $W^i_F$.
\end{rem}

\begin{example}
	Consider the Cartier module $F^s_*\omega_A \to \omega_A$ (see \autoref{ex:Cartier_operator}). Then note that for any non--trivial $L \in \Pic^0(A)$, we have $H^g(A, \omega_A \otimes L) = 0$ by Serre duality. Thus, if $L$ is not $p^e$--torsion for some $e$, then $L \notin W^g_F$.
	
	On the other hand, the morphism $H^g(A, F_*\omega_A) \to H^g(A, \omega_A)$ is Serre dual to the map taking $p^s$--powers $H^0(A, \cO_A) \to H^0(A, F^s_*\cO_A)$, which is always an isomorphism. Thus, $0 \in W^g_F$, so we conclude that \[ W^g_F = \set{x \in \bighat{A}}{x \mbox{ is $p^{es}$--torsion for some $e > 0$}} = \bigcup_{e > 0} (p^{es})^{-1}\{0_{\bighat{A}}\}.\]
\end{example}
	
	Even though the example above is rather basic, an important observation should be made: the set $W^g_F$ is \emph{never} closed if $A$ has positive $p$--rank, and it is even dense if $A$ is ordinary. Therefore it is a priori unclear what it means to ``bound its codimension'' as for the sets $V^i(\cM)$. In general, the idea will be to find a closed subset $W^i$ that ``approximates'' the locus $W^i_F$, in the sense that $W^i_F$ should be close to $\bigcup_e(p^{es})^{-1}W^i$. 
	
	For this specific example, we would have $W^g = \{0_{\bighat{A}}\}$. Unfortunately, we cannot always have \[ \bigcup_{e > 0}\: (p^{es})^{-1}W^i = W^i_F, \] in general (see \autoref{first_example}, \autoref{example_euler_char}), but the difference can be somewhat ``bounded'' by some smaller--dimensional closed subset $Z^i$. This is all encompassed in the following result:
	
\begin{thm}\label{generic vanishing}
	Let $\cM$ be a Cartier module on $A$. Then for all $0 \leq i \leq g$, there exist closed subsets $W^i$ and $Z^i$ of $\bighat{A}$ with the following properties:
	\begin{enumerate}[topsep=1ex, itemsep=1ex]
		\item\label{itm:gv_inclusions} $W^{i + 1} \inc W^i$, $Z^{i + 1} \inc Z^i$ and $Z^i \inc W^i$ for all $i$;
		\item\label{itm:gv_codim} $\codim W^i \geq i$, $\codim{Z^i} \geq i + 1$;
		\item\label{itm:gv_p_closed} $p^s(W^i) = W^i$;
		\item\label{itm:gv_link_with_section}
		we have \[ W^0 \inc V^0 = \set{\alpha \in \bighat{A}}{H^0(A, \cM \otimes \cP_{\alpha}) \neq 0}; \]
		\item\label{van_and_non_van} there are inclusions  \[\left(\bigcup_{e \geq 0}(p^{es})^{-1}W^i\right) \setminus \ttilde{Z^i} \: \inc \: W^i_F \: \inc \: \bigcup_{e \geq 0}(p^{es})^{-1}W^i,\] where \[ \ttilde{Z^i} \coloneqq \set{\alpha \in \bighat{A}}{p^{es}(\alpha) \in Z^i \mbox{ for infinitely many }e}. \]
		In particular, if $W^i$ has a component of codimension $i$, then very general points of $\bigcup_e(p^{es})^{-1}W^i$ lie in $W^i_F$.
	\end{enumerate}
\end{thm}

\begin{rem}
	The proof explains us how to construct these closed subsets $W^i$'s and $Z^i$'s explicitly, see \autoref{notation_V_inj_and_W_and_so_on}. Although the explicit construction is cumbersome, it turns out that understanding $W^0$ and the codimension $i$ components of the subsets $W^i$ is more doable. We will come back to this in \autoref{properties B^i}.
\end{rem}

The condition that $p^s(W^i) = W^i$ is the positive characteristic analogue of the fact that, in characteristic zero, the cohomological support loci are a finite union of torsion translates of abelian subvarieties. This analogy holds thanks to the following result:
\begin{thm}[{\cite{Pink_Roessler_Manin_Mumford_conjecture}}]\label{main thm Pink/Roessler}
	Assume that $A$ has no supersingular factor, and let $W \inc A$ be an irreducible closed subset such that $p^s(W) = W$. Then there exists $e > 0$, a point $a \in A$ of $(p^{es} - 1)$--torsion and an abelian subvariety $B \inc A$ such that \[ W = a + B. \]
\end{thm}
\begin{proof}
	By \cite[Theorem 3.1 and Remark 3.2]{Pink_Roessler_Manin_Mumford_conjecture} (see also \cite[Paragraph below Theorem 2.6]{Pink_Roessler_Conjecture_of_Beauville_and_Catanese_revisited}), we know that we can write $W = a' + B$ for some $a' \in A$ and some abelian variety $B$, so we are left to potentially change $a'$ to obtain the torsion property. By \cite[Proposition 6.1]{Pink_Roessler_Manin_Mumford_conjecture}, there exist $e > 0$ and a point $a \in W$ such that $p^{es}(w) = w$ (i.e. $w$ is $(p^{es} - 1)$--torsion). Indeed, with their terminology, $A$ is positive since the map $p^{es}$ is never an isomorphism on a positive--dimensional abelian variety). Since $B$ stabilizes $W$, we deduce that $a + B \inc W$. We then deduce by dimensional reasons that $a + B = W$, so the proof is complete.
\end{proof}

\begin{rem}
	\begin{itemize}
		\item In fact, \autoref{main thm Pink/Roessler} has a version that holds in greater generality, see \cite[Theorem 3.1]{Pink_Roessler_Manin_Mumford_conjecture}.
		\item The sets $W^i$ of \autoref{generic vanishing} may not be irreducible, but the closed subset in \autoref{main thm Pink/Roessler} is required to be irreducible. Nevertheless, if $W^i = X_1 \cup \dots \cup X_r$ denotes its irreducible decomposition, then for some $e > 0$, we also have $p^{es}(X_i) = X_i$ for all $i$. Thus, we can apply \autoref{main thm Pink/Roessler} on each $X_i$. 
		
		In particular, each closed subset $W^i$ in \autoref{generic vanishing} is a finite union of torsion translates of abelian subvarieties whenever $A$ has no supersingular factor.
	\end{itemize}
\end{rem}


%
%

We start by showing that $W^i_F$ can be computed in terms of $V$--modules (this can be very useful for actual computations).
\begin{lem}\label{SCSL for H^i = SCSL for Tor^i}
	Let $\cM$ be a Cartier module on $A$, and let $\cN \coloneqq \cH^0\FM_A(\cN)$ be the corresponding $V$--module.
	Then for all $i \geq 0$ and $\alpha \in \bighat{A}$, \[ \varprojlim H^i(A, F^{es}_*\cM \otimes \cP_{\alpha})  \cong \left(\colim \Tor_i(V^{es, *}\cN, k(- \alpha))\right)^{\vee}. \]
\end{lem}

\begin{proof}
	For all $e > 0$, we have a natural isomorphisms
	
	\[ H^i(A, F^{es}_*\cM \otimes \cP_\alpha) \expl{\cong}{\autoref{properties_Fourier_Mukai_transform}.\autoref{itm:more_gen_version}} \Tor_i(\FM_A(F^{es}_*\cM), k(-\alpha))^{\vee} \expl{\cong}{\autoref{properties_Fourier_Mukai_transform}.\autoref{itm:behaviour_pushforwards_and_pullbacks}} \Tor_i(V^{es, *}\FM_A(\cM), k(- \alpha))^{\vee}.\]

	Taking inverse limits gives \[ \lim H^i(A, F^{es}_*\cM \otimes \cP_\alpha^{\vee}) \cong \lim \Tor_i(V^{es, *}\FM_A(\cM), k(\alpha))^{\vee} \cong \left(\colim \Tor_i(V^{es, *}\FM_A(\cM), k(\alpha))\right)^\vee. \]
	
	Since $\cH^i\FM_A(\cM)$ is nilpotent for all $i \neq 0$ (\autoref{GV_Hacon_Pat}) and $\cN = \cH^0\FM_A(\cM)$, we have \[ \colim \Tor_i(V^{es, *}\FM_A(\cM), k(\alpha)) \cong \colim \Tor_i(V^{es, *}\cN, k(\alpha)), \] concluding the proof.
\end{proof}

As a consequence, the cohomological support loci can be completely understood in terms of commutative algebraic data of a module on $\bighat{A}$ (together with the $V$--module action). This is the key to prove \autoref{generic vanishing}. Let us state the following natural definition.

\begin{defn}
	Let $\cN$ be a $V$--module on $\bighat{A}$, and let $i \geq 0$. Define \[ W^i_V(\cN) \coloneqq \set{\alpha \in \bighat{A}}{\varinjlim \Tor_i(V^{es, *}\cN, k(-\alpha)) \neq 0} \inc \bighat{A}. \]
\end{defn}

 An immediate consequence of \autoref{SCSL for H^i = SCSL for Tor^i} is then the following:

\begin{cor}\label{equality_of_SCSL}
	Let $\cM$ be a Cartier module on $A$, and let $\cN \coloneqq \cH^0\FM_A(\cM)$ be the associated $V$--module. Then for all $i \geq 0$, we have \[ W^i_F(\cM) = W^i_V(\cN). \]
\end{cor}

The following construction will be useful to obtain the inclusions $p^s(W^i) \inc W^i$ in \autoref{generic vanishing}.

\begin{notation}\label{notation_lambda}
	For a subset $Z \inc \bighat{A}$, define \[ \Lambda(Z) \coloneqq \set{\alpha \in \bighat{A}}{p^{es}(\alpha) \in Z \esp \forall e \gg 0}. \]
\end{notation}

\begin{rem}\label{rem_easy_lambda}
	Note that if $p^s(Z)\inc Z$, then \[ \Lambda(Z) = \bigcup_{e \geq 0}(p^{es})^{-1}Z. \]
\end{rem}

\begin{lem}\label{p-dynamic closure of a subset}
	Let $Z \inc \bighat{A}$ be a closed subset. Then there exists a canonically induced closed subset $W_Z \inc Z$ such that $p^s(W_Z) = W_Z$ and \[ \Lambda(Z) = \Lambda(W_Z). \]
	
	This construction satisfies the following properties:
	\begin{enumerate}
		\item\label{explicit_1} if $p^s(Z) = Z$, then $W_Z = Z$;
		\item\label{explicit_2} more generally, if $p^s(Z) \inc Z$, then $W_Z = p^{es}(Z)$ for all $e \gg 0$;
		\item\label{inclusion} if $Z' \inc Z$ is an inclusion of closed subsets, then $W_{Z'} \inc W_Z$.
	\end{enumerate}
\end{lem}

\begin{proof}
	We will first construct a canonically induced closed subset $\overline{Y_Z}$ satisfying $p^s(\overline{Y_Z}) \inc \overline{Y_Z}$, and then construct $W_Z$ out of it.
	
	For all $\alpha \in \Lambda(Z)$, let $e(\alpha) \geq 0$ be the minimal integer such that $p^{es}(\alpha) \in Z$ for all $e \geq e(\alpha)$, and consider \[ Y_Z \coloneqq \set{p^{es}(\alpha)}{\alpha \in \Lambda(Z), \: e \geq e(\alpha)}.\] 
	Then we immediately have:
	\begin{itemize}
		\item  $p^s(Y_Z) \inc Y_Z$;
		\item  $Y_Z \inc Z$;
		\item $\Lambda(Y_Z) = \Lambda(Z)$.
	\end{itemize}
	Since $Z$ is closed, we then obtain that $\overline{Y_Z} \inc Z$, so $\Lambda(Y_Z) \inc \Lambda(\overline{Y}) \inc \Lambda(Z)$. Since $\Lambda(Y_Z) = W(Z)$, we deduce that $W(Y_Z) = W(\overline{Y_Z})$. Furthermore, we also have that $p^s(\overline{Y_Z}) \inc \overline{Y_Z}$, since $p^s(Y_Z) \inc Y_Z$. 
	
	Let us now construct $W_Z$. Consider the inclusions \[ \overline{Y_Z} \cni p^s(\overline{Y_Z}) \cni p^{2s}(\overline{Y_Z}) \cni \dots. \] This descending chain of closed subsets eventually stabilizes by Noetherianity; let us define $W_Z \coloneqq p^{es}(\overline{Y_Z})$ with $e \gg 0$. It is then immediate that $\Lambda(Z) = \Lambda(X_Z)$ and $p^s(W_Z) = W_Z$. 
	
	The three statements \autoref{explicit_1}, \autoref{explicit_2} and \autoref{inclusion} follow directly from our construction.
\end{proof}

Recall that given a subset $S \inc A$, we denote by $-S$ the image of $S$ under the inversion involution on $A$. We now have all the ingredients to prove \autoref{generic vanishing}.

\begin{proof}[Proof of \autoref{generic vanishing}]
	By \autoref{SCSL for H^i = SCSL for Tor^i}, letting $\cN = \cH^0\FM_A(\cM)$, recall that $W^i_F(\cM) = W^i_V(\cN)$. Let $\cN_{\injj}$ be the canonical injective $V$--module nil--isomorphic to $\cN$, i.e. $\cN_{\injj}$ is the image of the map $\cN \to \colim V^{es, *}\cN$ (see \autoref{canonical_injective_V_module}). Since the natural map $\colim V^{es, *}\cN \to \colim V^{es, *}\cN_{\injj}$ is an isomorphism by construction and $\varinjlim$ commutes with $\Tor_i$, we have \[ W^i_V(\cN) = W^i_V(\cN_{\injj}). \]
	We start with a first approximation \[ V^i_{\injj} := \set{\alpha \in \bighat{A}}{\Tor_i(\cN_{\injj}, k(-\alpha)) \neq 0} \] of $W^i_F(\cM)$. Let us list some properties of the subsets $V^i_{\injj}$:
	
	\begin{enumerate}
		\item $V^i_{\injj}$ is closed by \autoref{locus of non-zero Tor is closed}.
		\item $\codim V^i_{\injj} \geq i$. Indeed, for any scheme--theoretic point $\beta \in \bighat{A}$ of codimension strictly smaller than $i$ and any module $\cT$, we have \[ \Tor_i(\cT, k(\beta)) \cong \Tor_i^{\cO_{\bighat{A}, \beta}}(\cT_{\beta}, k(\beta)) = 0 \] since $\cO_{\bighat{A}, \beta}$ is regular of dimension less than $i$;
		\item $V^g_{\injj} \inc \dots \inc V^0_{\injj}$ by \autoref{meaning of vanishing of Tor}.
	\end{enumerate}
	Note that since $\cN \surj \cN_{\injj}$, we have that \[  V^0_{\injj} \expl{=}{Nakayama's lemma} -\Supp(\cN_{\injj}) \inc -\Supp(\cN) \expl{=}{\autoref{properties_Fourier_Mukai_transform}.\autoref{itm:support_H0}} V^0. \] 
	As the crucial property $p^s(V^i_{\injj}) = V^i_{\injj}$ is missing, we will now construct a closed subset $W^i \inc V^i_{\injj}$ such that $p^s(W^i) = W^i$. Note that given $\alpha \in \bighat{A}$, in order to have \[ \varinjlim \Tor_i(V^{es, *}\cN_{\injj}, k(-\alpha)) \neq 0, \] we must in particular have that $\Tor_i(V^{es, *}\cN_{\injj}, k(-\alpha)) \neq 0$ for all $e \gg 0$. By \autoref{Tor(V*) = 0 at x iff Tor = 0 at p(x)}, this is equivalent to the fact that $p^{es}(\alpha) \in V^i_{\injj}$ for all $e \gg 0$. In other words, \[ W^i_V \inc \Lambda(V^i_{\injj}) = \set{\alpha \in \bighat{A}}{p^{es}(\alpha) \in V^i_{\injj} \esp \forall e \gg 0}. \]
	Therefore we can apply \autoref{p-dynamic closure of a subset} to deduce the existence of closed subsets $W^i \coloneqq W_{V^i_{\injj}} \subseteq V^i_{\injj}$ such that
	
	\begin{itemize}
		\item $p^s(W^i) = W^i$;
		\item $W^g \inc \dots \inc W^0$ (see \autoref{p-dynamic closure of a subset}.\autoref{inclusion})
		\item $\codim W^i \geq i$;
		\item $W^i_V \inc \Lambda(V^i_{\injj}) =  \Lambda(W^i) = \bigcup_{e \geq 0}(p^{es})^{-1}W^i$ (see \autoref{rem_easy_lambda} for the last equality).
	\end{itemize}
	
	We have therefore proven \autoref{itm:gv_inclusions} for $W^i$, \autoref{itm:gv_codim} for $W^i$, \autoref{itm:gv_p_closed}, \autoref{itm:gv_link_with_section} and the right inclusion of \autoref{van_and_non_van}. It is therefore time to construct the closed subsets $Z^i$. Let $\cC$ denote the cokernel of $\cN_{\injj} \inj V^{s, *}\cN_{\injj}$. Note that if $\Tor_{i + 1}(\cC, k(-\alpha)) = 0$, then we have an injection \[ \Tor_i(\cN_{\injj}, k(-\alpha)) \inj \Tor_i(V^{s, *}\cN_{\injj}, k(-\alpha)) \] by the long exact sequence of Tor--groups. Note also that the cokernel of $V^{es, *}\cN_{\injj} \ra V^{(e + 1)s, *}\cN_{\injj}$ is simply $V^{es, *}\cC$. We then define \[ Z^i \coloneqq \set{\alpha \in \bighat{A}}{\Tor_{i + 1}(\cC, k(-\alpha)) \neq 0} \cap W^i. \] As before, $Z^i$ is closed of codimension at least $i + 1$, $Z^{i + 1} \subseteq Z^i$ and $Z^i \inc W^i$. \\
	
	Note that at this point, we proved everything except the left inclusion in point \autoref{van_and_non_van}. Let us do that now, by first showing that \[ \set{\alpha \in \bighat{A}}{p^{es}(x) \in W^i \setminus Z^i \mbox{ for } e \gg 0} \subseteq W^i_V. \]
	Let $\alpha \in \bighat{A}$ be such that $p^{es}(\alpha) \in W^i \setminus Z^i$ for any $e \gg 0$. Then \[ \Tor_i(V^{es, *}\cN_{\injj}, k(-\alpha)) \neq 0 \] and \[ \Tor_i(V^{es, *}\cN_{\injj}, k(-\alpha)) \inj \Tor_i(V^{(e + 1)s, *}\cN_{\injj}, k(-\alpha)) \] for $e \gg 0$ by our previous discussion, implying in particular that \[ \varinjlim \Tor_i(V^{es, *}\cN_{\injj}, k(-\alpha)) \neq 0. \] In other words, this gives that $\alpha \in W^i_V$, so we obtained our sought inclusion. Finally, given that $p^s(W^i) = W^i$, we have that \[ \set{\alpha \in \bighat{A}}{p^{es}(x) \in W^i \setminus Z^i \mbox{ for } e \gg 0} = \left(\bigcup_{e \geq 0}(p^{es})^{-1}W^i\right) \setminus \ttilde{Z^i}, \] where \[ \ttilde{Z^i} \coloneqq \set{\alpha \in \bighat{A}}{p^{es}(\alpha) \in Z^i \mbox{ for infinitely many }e}. \] The proof is complete.
\end{proof}

\begin{notation}\label{notation_V_inj_and_W_and_so_on}
	From now on, we will use the notations $V^i_{\injj}$ and $W^i$ as in the proof of \autoref{generic vanishing}. In other words, given a Cartier module $\cM$, we set \[ V^i_{\injj}(\cM) \coloneqq \set{\alpha \in \bighat{A}}{\Tor_i(\cN_{\injj}, k(-\alpha)) \neq 0} \] and \[ W^i(\cM) \coloneqq W_{V^i_{\injj}(\cM)}, \] (see \autoref{p-dynamic closure of a subset}), where $\cN_{\injj}$ is the canonical injective $V$--module associated to $\cM$ (i.e. $\cN_{\injj} = \cH^0(\FM_A(\cM))_{\injj}$, see \autoref{canonical_injective_V_module}). Also, if $\cC$ denotes the cokernel of $\cN_{\injj} \inj V^{s, *}\cN_{\injj}$, then \[ Z^i(\cM) \coloneqq \set{\alpha \in \bighat{A}}{\Tor_{i + 1}(\cC, k(-\alpha)) \neq 0}. \]
\end{notation}

\begin{rem}
	Note that the subsets of \cite[Theorem 1.1]{Hacon_Pat_GV_Geom_Theta_Divs} are precisely our subsets $V^i_{\injj}$. Hence, our $W^i$'s are more precise approximations of $W^i_F$. This is exactly the ``slight improvement'' we were referring to at the beginning of the section.
\end{rem}

In practice, working with the subsets $V^i_{\injj}$ can be more practical than with the subsets $W^i$. Some of their properties are listed below.

\begin{prop}\label{properties B^i}
	Let $\cM$ be a Cartier module on $A$, with associated injective $V$--module $\cN_{\injj}$. Then the following holds:
	
	\begin{enumerate}[itemsep=1ex]
		\item\label{inclusions} $V^g_{\injj} \subseteq \dots \subseteq V^0_{\injj}$;
		\item\label{codim} $\codim V^i_{\injj} \geq i$.
		\item\label{pB^0_is_B^0} $p^s(V^0_{\injj}) \subseteq V^0_{\injj}$, and hence $W^0 = p^{es}(V^0_{\injj})$ for $e \gg 0$;
		\item\label{non_vanishing_support} $V^0_{\injj} \inc V^0$.
		\item\label{codim_i_components} If $S^i$ denotes the union of the components of $V^i_{\injj}$ of codimension $i$, then $p^s(S^i) \subseteq S^i$. Hence, $p^{es}(S^i) \subseteq W^i$ for all $e \gg 0$, and it contains all components of $W^i$ of codimension $i$.
	\end{enumerate}
\end{prop}
\begin{proof}
	The statements \autoref{inclusions}, \autoref{codim} and \autoref{non_vanishing_support} were already shown during the proof of \autoref{generic vanishing}. The first part of \autoref{pB^0_is_B^0} is immediate, since $\cN_{\injj} \inj V^{s, *}\cN_{\injj}$ and $V^0_{\injj} = -\Supp(\cN_{\injj})$ by Nakayama's lemma. To obtain that $W^0 = p^{es}(V^0_{\injj})$, use \autoref{p-dynamic closure of a subset}.\autoref{explicit_2}.
	We now prove \autoref{codim_i_components}. Let $\alpha \in V^i_{\injj}$ be a scheme--theoretic point of codimension $i$. Then \[ \Tor_i(\cN_{\injj}, k(-\alpha)) \inj \Tor_i(V^{s, *}\cN_{\injj}, k(-\alpha))\] (there is no $\Tor_{i+ 1}$ on an $i$--dimensional regular local ring), so $\Tor_i(\cN_{\injj}, k(p^s(-\alpha)) \neq 0$ by \autoref{Tor(V*) = 0 at x iff Tor = 0 at p(x)}. In particular, $p^s(\alpha) \in V^i_{\injj}$, so we have $p^s(S^i) \subseteq S^i$. The rest follows as in the proof of \autoref{pB^0_is_B^0}.
\end{proof}

\subsection{Examples}\label{sec:examples}
 
Let us give a few examples, giving an idea of what these subsets are, and illustrating pathologies which can occur.

\begin{example}\label{first_example}
	Let $r$ denote the $p$--rank of $A$. If $r < g$, then by definition, the Frobenius action on $H^r(A, \cO_A)$ is not nilpotent, but the action on $H^{r + 1}(A, \cO_A)$ is nilpotent. By Serre duality, we then deduce that $0 \in W^{g - r}_F(\omega_A)$ but $0 \notin W^{g - r + 1}_F(\omega_A)$. Thus, $W^{g - r}_F(\omega_A)  \not\subseteq W^{g - r + 1}_F(\omega_A)$. In particular, although we have \[ W^g \inc W^{g - 1} \inc \dots \inc W^0 \] and \[ V^g_{\injj} \inc V^{g - 1}_{\injj} \dots \inc V^0_{\injj} \] in general, this does not necessarily hold for the subsets $W^i_F$ (when $r < g$). We will nevertheless see in \autoref{inclusions_SCSL} that these inclusions ``essentially hold'' when $A$ is ordinary (i.e. when $r = g$).
	
	In this example, we have the following:
	\begin{align*}
		& \cH^0(\FM_A(\omega_A))_{\injj} = \FM_A(\omega_A) = k(0); \\
		& W^i = V^i_{\injj} = \{0_{\bighat{A}}\} \mbox{ for all } 0 \leq i \leq g; \\
		& \ttilde{Z^i} = \begin{cases}
			\emptyset, \: \mbox{ if } i = g \mbox{ or } A \mbox{ is ordinary;}\\
			\bighat{A}[p^{\infty}], \: \mbox{ else;}
		\end{cases}\\
	    & W^i_F = \begin{cases}
			\emptyset, \: i \leq g - r; \\
			\bighat{A}[p^{\infty}], \: i > g - r,
		\end{cases}
	\end{align*} 
	where $\bighat{A}[p^{\infty}]$ denotes the subset of $p$--power torsion points of $\bighat{A}$. In particular, whenever $0 < r < g$, we see that both inclusions in \autoref{generic vanishing}.\autoref{van_and_non_van} can be strict.
\end{example}

\begin{example}\label{example_euler_char}
		Let $E$ be an ordinary elliptic curve, and assume for simplicity that $p \neq 2$. Although $E \cong \bighat{E}$, we will still write $E$ and $\bighat{E}$ to distinguish Cartier modules on $E$ and $V$--modules on $\bighat{E}$.
		
		Since we will work with divisors and the group law at the same time, we will use the following notation:
		
		\begin{itemize}
				\item the notation for divisors is the same as usual, i.e. we use the additive notation;
				\item the neutral element of $E$ is denoted $e$;
				\item the group law on $E$ is written multiplicatively. Recall that by definition of the group law, we have the following linear equivalence of \emph{divisors} for all $\alpha, \beta \in \bighat{E}$:
				\[ \alpha + \beta \sim \alpha\beta + e. \]
		\end{itemize}
	
		Consider the line bundle $L \coloneqq \cO_{\bighat{E}}(e)$. Then \[ V^{s, *}L \cong \cO_{\bighat{E}}\left(\sum_{p^s(\alpha) = e} \alpha\right) \cong \cO_{\bighat{E}}\left(\left(\prod_{p^s(\alpha) = e}\alpha\right) + (p^s - 1)e\right) \expl{\cong}{$p \neq 2$} \cO_{\bighat{E}}(p^se). \]
		
		We indeed used that $p \neq 2$ to infer that for any non--trivial $\alpha$ of $p^s$--torsion, $\alpha^{-1}$ is also of $p^s$--torsion and $\alpha \neq \alpha^{-1}$ since $\alpha$ is not $2$--torsion. Therefore $\alpha$ is cancelled by $\alpha^{-1}$ and this product is indeed $e$.
		
		Now, let $\theta \colon L \to V^{s, *}L$ be the (injective) morphism corresponding to the section $1 \in H^0(\bighat{E}, \cO_E((p^s - 1)e))$. Then $\theta \otimes k(0) = 0$, so \[ \colim V^{es, *}L \otimes k(0) = 0. \] On the other hand, for any $\alpha \in \bighat{E}$ not of $p$--power torsion, each morphism $V^{es, *}\theta$ is an isomorphism at $\alpha$.
		
		Now, let $\cM = \cH^0\FM_{\bighat{A}}(L)$ be the Cartier module on $E$ corresponding to $L$ (see \autoref{equivalence_V_crystals_and_Cartier_crystals}). Then our computation above shows that \[ W^0_F(\cM) = \bighat{E} \: \setminus \: \bigcup_{e > 0} (p^{es})^{-1}\{0_{\bighat{E}}\} = \bighat{E} \setminus \bighat{E}[p^{\infty}], \] that $W^0 = V^0_{\injj} = \bighat{E}$ ($L$ is an injective $V$--module), and $Z^0 = \{0_{\bighat{E}}\} \neq \emptyset$ (so $\ttilde{Z^0} = \bighat{E}[p^{\infty}]$). In particular, this shows that the answer to \cite[Question 1.14.(b)]{Hacon_Patakfalvi_Zhang_Bir_char_of_AVs} is negative.
\end{example}
\begin{rem}[Euler characteristics]
	It is now well--known that if $\cM$ is a GV--sheaf on an abelian variety $A$, then $\chi(A, \cM) \geq 0$: the reason is that if $\alpha \in \bighat{A}$ is general, then \[ \chi(A, \cM) = \chi(A, \cM \otimes \cP_{\alpha}) = h^0(A, \cM \otimes \cP_{\alpha}) \geq 0. \]
	
	With the idea that ``Cartier modules are GV--sheaves up to nilpotence'', one can wonder if the following holds: let $\cM$ be a Cartier module on an abelian variety $A$, then $\chi_{ss}(A, \cM) \geq 0$ (see \autoref{rem:def_Cartier_mod}.\autoref{itm:notation H^0_ss} for this notation).
	
	Unfortunately, this fails in general: if we take $r = 0$ in \autoref{first_example}, then $\chi_{ss}(A, \omega_A) = (-1)^g$, which is negative whenever $g$ is odd. The issue why the characteristic zero argument does not go through is that semistable Euler characteristic is \emph{not} invariant under twists. In \emph{loc. cit.}, we have that $\chi_{ss}(A, \omega_A \otimes \cP_{\alpha}) = 0$ if $\alpha$ is not of $p$--power torsion, so this semistable Euler characteristic can go both up or down at ``special points''. 
	
	This example only happens when $r = 0$, so one might hope that this pathological behaviour (the failure of invariance under twists) does not occur when $A$ is for example ordinary. Unfortunately, \autoref{example_euler_char} shows that even this is not true: with $\cM$ as in \emph{loc. cit.}, then our computations show that \[ \chi_{ss}(E, \cM \otimes \cP_\alpha) = \begin{cases}
		0 \mbox{ if $\alpha$ is $p$--power torsion;} \\
		1 \mbox{ if not.}
	\end{cases} \]

	In any case, it can be shown that $\chi_{ss}(A, \cM) \geq 0$ for any Cartier module on an ordinary abelian variety. This is explored in \cite{Baudin_On_the_Euler_characteristic_of_weakly_ordinary_varieties_of_maximal_albanese_dimension}.
	
\end{rem}

\section{Effectivity of $\omega_X$ for varieties of maximal Albanese dimension}

Our goal here is to show \autoref{intro_effectivity_canonical}. Although the first part of this statement (the one which does not mention ordinarity) already follows from \autoref{generic vanishing} as we will see, the second part will require some more work. \\

Let us first show what we claimed in \autoref{first_example}.

\begin{slem}\label{inclusions_SCSL}
	Assume that $A$ is ordinary, and let $\alpha \in \bighat{A}$ be a torsion point. Then for any Cartier module $\cM$ on $A$ and $i \geq 0$, then \[ \alpha \in W^{i + 1}_F(\cM) \implies \alpha \in W^i_F(\cM). \]
\end{slem}

\begin{proof}
	Let $\cN \coloneqq \cH^0\FM_A(\cM)$. Since $W^j_F(\cM) = W^j_V(\cN)$ for all $j \geq 0$ by \autoref{equality_of_SCSL}, we have to show that if $-\alpha \in W^{i + 1}_V(\cN)$ is a torsion point, then also $-\alpha \in W^i_V(\cN)$.

	Let $\beta \coloneqq -\alpha$ be a torsion point, and let $m \geq 0$ be the order of the point $\beta$ in the group law of $\bighat{A}$. Since $m$ divides $p^{as}(p^{bs} - 1)$ for some $a, b \geq 1$, and since $(p^s)^{-1}W^i_F(\cM) = W^i_F(\cM)$, we may replace $\beta$ by $p^{as}(\beta)$ and assume that $p^{bs}(\beta) = \beta$. Given that \[ \colim_e V^{es, *}\cN \cong \colim_e V^{ebs, *}\cN, \] we may replace $s$ by $bs$ and assume that $b = 1$. In other words, we have $p^s(\beta) = \beta$. In particular, $V^s$ induces a morphism $\cO_{\bighat{A}, \beta} \to \cO_{\bighat{A}, \beta}$. Write $N \coloneqq \cN_\beta$ (i.e. the stalk at $\beta$), and let $\fm_\beta$ denote the maximal ideal of $\cO_{\bighat{A}, \beta}$. Now, let \[ 0 \to K \to P \to N \to 0 \] be a short exact sequence of $\cO_{\bighat{A}, \beta}$--modules, where $P$ is free of finite rank and $P \otimes k(\beta) \to N \otimes k(\beta)$ is an isomorphism. By freeness of $P$, there exists a $V$--module stucture $\theta_P$ on $P$ such that $P \surj N$ becomes a morphism of $V$--modules. In particular, $K$ acquires the structure of a $V$--module too, and we have a commutative diagram with exact rows as follows: \[ \begin{tikzcd}
		0 \arrow[r] & K \arrow[r] \arrow[d, "\theta_K"] & P \arrow[r] \arrow[d, "\theta_P"] & N \arrow[r] \arrow[d, "\theta_N"] & 0  \\
		0 \arrow[r] & {V^{s, *}K} \arrow[r] & {V^{s, *}P} \arrow[r] & {V^{s, *}N} \arrow[r]      & 0
	\end{tikzcd} \]
	For all $j \geq 1$, we then have isomorphisms \[ \colim \Tor_j(V^{es, *}K, k(\beta)) \cong \colim \Tor_{j + 1}(V^{es, *}N, k(\beta)). \] Since $P \to N$ is an isomorphism at $k(\beta)$, this is also the case for each map $V^{es, *}P \to V^{es, *}N$. In particular, we also have that \[ \colim V^{es, *}K \otimes k(\beta) \cong \colim \Tor_1(V^{es, *}N, k(\beta)). \]
	This shows that it is enough to prove that if $\beta \in W^1_F(\cN)$, then $\beta \in W^0_F(\cN)$ (for the other $i$'s, replace $N$ by $K$ and so on).
	We are going to show the contrapositive, so assume that $\beta \notin W^0_F$, i.e. \[ \varinjlim V^{es, *}N \otimes k(\beta) = 0. \]
	This implies that for some $e > 0$, $\theta_N^e(N) \inc \fm_\beta(V^{es, *}N)$. Since both $P \to N$ and $V^{es, *}P \to V^{es, *}N$ are isomorphisms after tensoring by $k(\beta)$, we also have that $\theta_P^e(P) \inc \fm_\beta(V^{es, *}P)$. 
	
	By Krull's theorem, there exists $c \geq 1$ such that $\fm_\beta^cP \cap K \inc \fm_\beta K$. Thus, 
	\begin{align*} 
		\theta_K^{ce}(K) & \inc V^{ces, *}K \cap \theta_P^{ce}(P) \\
		& \inc V^{ces, *}K \cap \fm_\beta^{c}(V^{ces, *}P) \\
		& \expl{=}{$V$ is étale by \autoref{basic facts about AV}.\autoref{itm:ordinary iff Verschiebung etale}, so $V(\fm_{\beta}^c)\cO_{\bighat{A}, \beta} = \fm_{\beta}^c$} V^{ces, *}K \cap V^{ces, *}(\fm_\beta^{c}P) \\
		& \expl{=}{\cite[Theorem 7.4.(i)]{Matsumura_Commutative_Ring_Theory}} V^{ces, *}\left(K \cap \fm_\beta^{c}P\right) \\
		& \inc V^{ces, *}(\fm_\beta K) \\
		& \expl{=}{$V$ is étale} \fm_\beta V^{ces, *}K.
	\end{align*}
	In other words, $\theta_K^{ce} \otimes k(\beta) = 0$. But then, $V^{r, *}\theta_K^{ce} \otimes k(\beta) = 0$ for all $r \geq 0$, which gives that \[ 0 = \varinjlim V^{es, *}K \otimes k(\beta) \cong \varinjlim \Tor^1(V^{es, *}N, k(\beta)).  \] We conclude that $\beta \notin W^1_F(\cN)$, completing the proof.
\end{proof}

Now let us move to the geometric consequence. We start with a useful well--known fact.

\begin{slem}\label{regular_implis_simple}
	Let $Y$ be a regular variety. Then $\omega_Y$ is simple as a Cartier module (i.e. any sub--Cartier module is either $\omega_Y$ or zero).
\end{slem}
\begin{proof}
	Since $Y$ is Gorenstein, we have to show that the test ideal $\tau(\cO_Y)$ of $Y$ is $\cO_Y$ (see \cite[Section 3]{Schwede_Tucker_A_survey_of_test_ideals}). In other words, we need to show that $Y$ is strongly $F$--regular. This holds by \cite[Theorem 3.1]{Hochster_Huneke_Tight_closure_and_strong_F_reg}.
\end{proof}

We can finally show our last main theorem.

\begin{sthm}\label{effectivity_canonical}
	Let $X$ be a normal proper variety of maximal Albanese dimension. Then $H^0(X, \omega_X) \neq 0$. If furthermore $X$ admits a generically finite morphism to an ordinary abelian variety, then $S^0(X, \omega_X) \neq 0$. 
\end{sthm}
\begin{srem}\label{last_rem} 
	If $a \colon X \to A$ is a generically finite morphism to an abelian variety which is separable onto its image, then the fact that $H^0(X, \omega_X) \neq 0$ is much easier to prove (and also well--known). Indeed, let $Z \coloneqq a(X) \inc A$. Since $\Omega_A^1$ is globally generated and surjects onto $\Omega^1_Z$, we obtain that $\Omega^1_Z$ is globally generated. By separability of $a$, the pullback map $a^*\Omega_Z^1 \to \Omega_X^1$ is generically surjective, so $\Omega_X^1$ is generically generated by global sections. In particular, we obtain that $H^0(X, \omega_X) \neq 0$. Note that one can also adjust this proof to obtain that $S^0(X, \omega_X) \neq 0$ if $A$ is ordinary, since then $\Omega_A^1$ admits a basis of closed $1$--forms fixed under the Cartier operator (which is compatible with wedges and pullbacks).
	
	The main issue is when $a$ is not separable onto its image, as then this approach cannot work and we must proceed differently.
\end{srem}
\begin{proof}
	We will show that the conclusion of \autoref{effectivity_canonical} holds for \emph{any} Cartier submodule of $\omega_X$. Hence, let $\ttilde{\omega_X}$ be a Cartier submodule of $\omega_X$. Set also $d \coloneqq \dim(X)$. Let $a \colon X \to A$ be a generically finite morphism to an abelian variety $A$. 
	
	Assume first that $a \colon X \to A$ is finite. Note that by \cite[Proposition 2.4]{Patakfalvi_Zdanowicz_Ordinary_varieties_with_trivial_canonical_bundle_are_not_uniruled}, Serre duality induces an isomorphism  on $H^d(X, \omega_X) \cong H^0(X, \cO_X)^{\vee}$. Since the Frobenius action on $H^0(X, \cO_X)$ is non--nilpotent, we deduce that $H^d_{ss}(X, \omega_X) \neq 0$. Consider the short exact sequence of Cartier modules \[ 0 \to \ttilde{\omega_X} \to \omega_X \to \cC \to 0. \]
	Since $X$ is regular in codimension one, we deduce from \autoref{regular_implis_simple} that $\codim (\Supp \cC) \geq 2$. Hence, the natural morphism $H^d(X, \ttilde{\omega_X}) \to H^d(X, \omega_X)$ is an isomorphism, whence $H^d_{ss}(X, \ttilde{\omega_X}) \neq 0$. Given that $a$ is finite, we obtain that $H^d_{ss}(A, a_*\ttilde{\omega_X}) \cong H^d_{ss}(X, \ttilde{\omega_X}) \neq 0$. In other words, \[ 0 \in W^d_F(a_*\ttilde{\omega_X}) \expl{\subseteq}{\autoref{generic vanishing}.\autoref{van_and_non_van}} \bigcup_{e > 0} (p^e)^{-1}W^d, \] so $0 \in W^d$.
	Now, we have \[ W^d \expl{\inc}{\autoref{itm:gv_inclusions}} W^0 \expl{\inc}{\autoref{itm:gv_link_with_section}} \set{\alpha \in \bighat{A}}{H^0(A, a_*\ttilde{\omega_X} \otimes \cP_{\alpha}) \neq 0}, \] where both items refer to \autoref{generic vanishing}. In particular, we deduce that \[ H^0(X, \ttilde{\omega_X}) \neq 0. \] If we further assume that $A$ is ordinary, then since $0 \in W^d_F(a_*\ttilde{\omega_X})$, we directly obtain that $0 \in W^0_F(a_*\ttilde{\omega_X})$ by \autoref{inclusions_SCSL}. In other words, $H^0_{ss}(X, \ttilde{\omega_X}) \neq 0$. \\
	
	Now, let us prove the statement in the general case, i.e. we do not assume that $a$ is finite anymore. Let \[ \begin{tikzcd}
		X \arrow[rr, "f"] \arrow[rrrr, "a"', bend right] &  & Y \arrow[rr] &  & A.
	\end{tikzcd} \] denote the Stein factorization of $a$. Then $f_*\ttilde{\omega_X} \inc f_*\omega_X \inc \omega_Y$, so the statement for $\ttilde{\omega_X}$ on $X$ follows from the statement for the Cartier sub--module $f_*\ttilde{\omega_X} \inc \omega_Y$ on $Y$.
\end{proof}

\bibliographystyle{alpha}
\bibliography{Bibliography}

\Addresses

\end{document}